\documentclass{amsart}
\usepackage{amssymb}
\usepackage{amsmath}
\usepackage{amsfonts}
\usepackage{amsthm}

\newtheorem{prop}{Proposition}
\newtheorem{lem}{Lemma}
\newtheorem{cor}{Corollary}
\newtheorem{thm}{Theorem}

\newcommand{\R}{{\mathbb R}}
\renewcommand{\H}{{\mathfrak H}}
\newcommand{\PSL}{\mathrm {PSL}}
\newcommand{\GL}{\mathrm {GL}}

\newcommand{\bs}{\backslash}
\newcommand{\inv}{^{-1}}
\newcommand{\f}{\frac}
\renewcommand{\(}{\left(}
\renewcommand{\)}{\right)}
\renewcommand{\{}{\left \lbrace}
\renewcommand{\}}{\right \rbrace}
\renewcommand{\bar}{\overline}
\newcommand{\beq}{\begin{equation}}
\newcommand{\eeq}{\end{equation}}
\newcommand{\bmx}{\( \begin{matrix}}
\newcommand{\emx}{\end{matrix} \)}

\newcommand{\dist}{\mathrm{dist}}
\newcommand{\Id}{\mathrm{Id}}
\newcommand{\len}{\mathrm{len}}
\renewcommand{\asymp}{\sim}

\newcommand{\twocase}[5]{#1 \begin{cases} #2 & \text{#3}\\ #4
&\text{#5} \end{cases}   }

\begin{document}

\markboth{K. Martin, M. McKee and E. Wambach}{A relative trace formula for a compact Riemann surface}

\title{A relative trace formula for a compact Riemann surface}

\author{Kimball Martin}
\address {Department of Mathematics, University of Oklahoma\\
Norman, Oklahoma~73019-0315~USA}
\email{kmartin@math.ou.edu}

\author{Mark McKee}
\address{Department of Mathematics, University of Oklahoma\\
Norman, Oklahoma~73019-0315~USA}
\email{mmckee@math.ou.edu}

\author{Eric Wambach}
\address{Department of Mathematics, Indian Institute of Technology
Bombay \\Mumbai~400 076~India}
\email{wambach@math.iitb.ac.in}

\thanks{This is a corrected
version of the printed article (\emph{Int.\ J.\ Number Theory 7 (2011), no.\ 2, 389--429}) as of
Feb.\ 2, 2012.  We thank Masao Tsuzuki
and Roelof Bruggeman for bringing errors in that version to our attention.}

\maketitle

\begin{abstract}
We study a relative trace formula for a compact Riemann surface with respect to a
closed geodesic $C$.  This can be
expressed as a relation between the period spectrum and the ortholength spectrum of
$C$.  This provides a new proof of asymptotic results for both the periods of Laplacian eigenforms along $C$ as well estimates on the lengths of geodesic segments which start and 
end orthogonally on $C$.  Variant trace formulas also lead to several simultaneous 
nonvanishing results for different periods.
\end{abstract}



\section{Introduction}

Let $\H$ be the upper half plane, $\Gamma$ a discrete cofinite subgroup of 
$G=\PSL_2(\R)$ and $X = \Gamma \bs \H$ be the quotient space.  
Let $\Delta$ denote the hyperbolic Laplacian on $X$, and $\{ \phi_n \}$
be an orthonormal basis of $\Delta$-eigenfunctions for the discrete spectrum of
$L^2(X)$.

In the case of $\Gamma=\PSL_2(\mathbb Z)$, 
Kuznetsov \cite{Kuz} and Bruggeman \cite{Brug} independently derived a formula, known as
the Kuznetsov(-Bruggeman) (sum or trace) formula, which essentially relates the Fourier 
coefficients of the $\phi_i$'s with  Kloosterman sums $S(n,m,c)$.  
As in the case of the Selberg trace formula, both sides of the formula involve a test 
function.  This implies estimates on sums of Fourier coefficients as well as estimates
on sums of Kloosterman sums.  The Kuznetsov formula has been generalized to other
$\Gamma$, but we will not go into that here; see \cite{Iwa} for more details.

The Kuznetsov trace formula is similar to the Selberg trace formula and this is perhaps
most clearly seen from the point of view of Jacquet's relative trace formula.  For an
appropriate test function $f: G \to \R$, one forms an associated kernel 
$K(x,y): \Gamma \bs G \times \Gamma \bs G \to \R$ which is
\[ K(x,y) = \sum_{\gamma \in \Gamma} f(x\inv \gamma y) = \sum_n K_{\phi_n}(x,y) + 
K_{\text{Eis}}(x,y), \]
where the terms on the right encode certain spectral  
information.  The left hand expansion is known as the
geometric expansion and the right hand side is the spectral expansion.  
The Selberg trace formula comes by integrating each expansion over the diagonal
$\Gamma \bs G \subset \Gamma \bs G \times \Gamma \bs G$.
Let $N$ be a unipotent subgroup of $G$.  When $\Gamma_N = \Gamma \cap N$ is nontrivial,
integrating the two expansions for the kernel over $\Gamma_N \bs N \times \Gamma_N \bs N$
essentially gives the Kuznetsov formula.

Now one can of course integrate $K(x,y)$ over other product subgroups, yielding what are
now called relative trace formulas, and one particular case of 
interest is the following.  Let $H$ be a torus of $G$ and suppose $\Gamma_H = \Gamma \cap H$
is nontrivial.  Then integrating $K_{\phi_n}(x,y)$ over 
$\Gamma_H \bs H \times \Gamma_H \bs H$ yields the square of a period integral of $\phi_n$.
This was studied in detail in the adelic context in \cite{Jac1},
\cite{Jac2} in connection with $L$-values.

The relative trace formula has proved to be a powerful tool in studying periods and 
$L$-functions.  However, little has been done, except in the classical Kuznetsov case, to 
understand the role of the relative trace formula in spectral geometry, especially in light
of the applications of the Selberg trace formula in this field.  We propose to study here
the classical analogue of the relative trace formula in \cite{Jac1}, for simplicity, in
the case where $X$ is compact.  This formula provides a new relation between what might
be called the period spectrum and the ortholength spectrum relative to a fixed
closed geodesic $C$ on $X$.  Consequently we give a new proof of some asymptotic results
for these periods and ortholengths.  We hope this will be of interest both from the point of
view of the study of compact Riemann surfaces, and from the point of view of better 
understanding the relative trace formula.  Now we describe our work in greater detail.
To put it in context, we recall some facts about the Selberg trace formula.

\medskip
Suppose $X=\Gamma \bs \H$ be a smooth compact (hyperbolic) Riemann surface.
Let $\{ \phi_n \}$ be an orthonormal basis of eigenforms of $\Delta$
in $L^2(X)$ with corresponding eigenvalues $\lambda_n$.

The Selberg trace formula is a powerful tool for studying the spectral
geometry of $X$ (e.g., \cite{McK}, \cite{Hej}).  Instead of working with a kernel on
$\Gamma \bs G$ as mentioned above, we will work directly with a kernel on $X$, as in 
\cite{Hej}.  For a suitable test function $\Phi$ one associates a kernel $K: X \times
X \to R$, and integrates $K$ over the diagonal $X \subset X \times X$ to get the
Selberg trace formula.  This formula relates a
sum over conjugacy classes of $\Gamma$ of orbital integrals (the geometric side)
with the eigenvalues $\lambda_n$ (the spectral side).
The nontrivial conjugacy classes of $\Gamma$ are in one-to-one correspondence
with classes of closed geodesics on $X$, and the geometric side may be expressed
in terms of the volume of $X$ and the lengths of shortest primitive closed geodesics on
$X$.  The collection of these lengths is called the length spectrum of $X$, so
the Selberg trace formula relates the length spectrum with the eigenvalue spectrum
$\{ \lambda_n \}$ of $X$.

By varying appropiate test functions $\Phi$, one is able to use the trace formula to
get many beautiful results.  We mention two.  The first is
Weyl's law:
\[ \#\{ \lambda_n \leq x \} \sim \f{\mathrm{vol}(X)}{4\pi} x \text{ as } x\to \infty. \]
The second is the Prime Geodesic Theorem:  if $\pi_0(x)$ denotes the number
of closed oriented geodesics $\gamma$ with norm $N(\gamma) = e^{\mathrm{len}(\gamma)}$
less than $x$,
then
\[ \pi_0(x) \sim \f{x}{\log x}. \]

The setup for our relative trace formula is as follows.  Fix a closed
geodesic $C$ on $X$. 
We integrate the kernel $K$ on the product subspace $C \times C$
of $X \times X$ against a ``character'' $\chi$ of $C$ (cf. Section \ref{sec:twistper}).  
This is the compact classical (non-adelic) version of the 
$\GL(2)$ relative trace formula in \cite{Jac1}.
The spectral side of this trace formula is
\[ \sum h(r_n)|P_\chi(\phi_n)|^2 \]
where $h$ is the Selberg transform of $\Phi$, $\lambda_n = \f 14 + r_n^2$ and 
$P_\chi(\phi)$ denotes the twisted period
\[ P_\chi(\phi) = \int_C \phi(t)\chi(t)dt. \]
We may think of these twisted periods as Fourier coefficients of $\phi$ along $C$.
In analogy with the Selberg trace formula, we will call the sequence $\{ P_\chi(\phi_n)\}$
the ($\chi$-){\em period spectrum} for $C$.

The geometric side is a sum over double cosets $\Gamma_C \bs \Gamma / \Gamma_C$
of {\em (relative) orbital integrals}, where $\Gamma_C$ is the stabilizer of
$C$ in $\Gamma$.  The nontrivial double cosets correspond to the finitely many
self-intersection points on $C$ and the {\em orthogonal spectrum} of $C$, by which
we mean the set of geodesic segments which start and end on $C$ and meet $C$ 
orthogonally at both ends.  One can order these geodesic segments
by length, and we call the sequence of these lengths
the  {\em (real) ortholength spectrum} of $C$.
In Section \ref{secthree}, we express the geometric side in terms of 
the length of $C$, its angles of self-intersection, and its ortholength spectrum.
(For simplicity, we only write this formula down for a general test function when $\chi$
is trivial, though the case of $\chi$ nontrivial is similar.)

Hence we discover that the relative trace formula provides a relation between
the period spectrum and the ortholength spectrum of $C$.
The analogue of the Selberg trace formula applications mentioned above 
would be asymptotics for the period and ortholength spectra.
In the present paper, we use this relative trace formula to obtain average asymptotics 
and various nonvanishing results for the periods, as well as asymptotic bounds for the 
ortholength spectrum.  This trace formula also suggests many questions beyond the
present work.  An obvious one is how precise can we make these asymptotics?
But other kinds of questions naturally arise also: for instance, do the
ortholength and period spectra determine each other, as is the case with the
length and eigenvalue spectra?  (It is known they do not determine the surface \cite{Zel2}.)

\medskip
Let us first discuss the period asymptotics.
In the case where $\Gamma$ is an arithmetic group, these periods are
related to special values of $L$-functions and Fourier coefficients of
modular forms (e.g., \cite{Wal}).  Less is known about them in the
non-arithmetic case.  For a general (cocompact) $\Gamma$, quantum ergodic
results (\cite{Shn}, \cite{CdV}, \cite{Zel}) state that the functions
$\phi_n$ become equidistributed with respect to an area measure, except
for a possible exceptional (density 0) subsequence.  In the case of
hyperbolic surfaces, a conjecture of Rudnick and Sarnak \cite{RS} asserts
that this thin subsequence should be empty.  (See \cite{HS} for remarkable 
recent work when $\Gamma = \mathrm{SL}_2(\mathbb Z)$.)  In a similar vein, one
expects the periods $P(\phi_n) = \int_C \phi_n \to 0$.  
In fact, one can show is the asymptotic
\begin{equation} \label{finasym}
\sum_{\lambda_n \leq x} |P(\phi_n)|^2 \asymp \f{\len(C)}{\pi} \sqrt{x}.
\end{equation}
Note that (\ref{finasym}) says that, {\em on average}, $P(\phi_n)$ is about 
$\lambda_n^{-1/4}$.  This of course implies that, apart from a possible exceptional
subsequence, $P(\phi_n) \to 0$. 

The equation (\ref{finasym}) seems to have first been stated in \cite{Hej2} by analyzing 
Dirichlet series and using the automorphic Green's kernel, however no proofs are given.
Subsequently, Good \cite{Good} extended Kuznetsov's formula to the case of 
Fourier coefficients along geodesics (as well as other cases), for both $\Gamma$ compact
and non-compact.  A consequence is an asymptotic of the above form.  Unfortunately,
\cite{Good} is rather difficult to penetrate, and it seems this work is not well understood.
(See also \cite{Good2}.)  Additionally, the Fourier coefficients in \cite{Good} are
not exactly the periods appearing in (\ref{finasym})---the asymptotic in 
\cite[Theorem 2]{Good} is of different order!
Later Zelditch
\cite{Zel2}, then Ji and Zworski \cite{JZ}, obtained analogues of (\ref{finasym}) 
in a general 
context using the wave kernel.  Recently, \cite{Pitt} has proved a formula similar
to the one in \cite{Good} for compact Riemann surfaces and again obtains (\ref{finasym}),
and \cite{Tsuzuki} proves an analogue of (\ref{finasym}) for some compact and non-compact
arithmetic $\Gamma$ using Green's functions.

Our first application of the relative trace formula is a new proof of (\ref{finasym}). 
To do this, we observe that the {\em natural} choice for a test function in our
relative trace formula is
$\Phi(x) = e^{-tx}$, which appears to have not been considered previously.
The reason for this choice of $\Phi$, is that in general the geometric side involves 
some rather complicated elliptic integrals, but for $\Phi(x) = e^{-tx}$ they degenerate 
into $K$-Bessel functions.  We work out the trace formula in this case in
Section \ref{secfour} (for simplicity still with $\chi$ trivial), and see it 
yields a rather striking limit formula:
\begin{equation}
\label{limform}
 \lim_{t \to \infty} e^t \sum K_{ir_n}(t) |P(\phi_n)|^2 =
\f 12 \textrm{length}(C).
\end{equation}
Since each $K_{ir_n}(t) \asymp \sqrt{\f \pi{2t}}e^{-t}$, this immediately
implies that an infinite number of $P(\phi_n)$'s are nonvanishing.  As
this is related to nonvanishing of $L$-values in the arithmetic case,
this seems to be a nontrivial statement, despite its apparent simplicity.

In order to get (\ref{finasym}), we would like to use
a Tauberian argument.  While estimates for $K_{ir}(t)$ in either 
$r$ or $t$, keeping the other fixed, are classical and well known, in our situation
one needs uniform estimates for $K_{ir}(t)$ when {\em both} $r$ and $t$ are 
allowed to vary, which is a much more subtle problem. 
These estimates, which seem new, are carried out in Section \ref{marksec}.  This
allows us to conclude (\ref{finasym}) (see Theorem \ref{thm2} below).

In Section \ref{sec:twistper}, we consider the relative trace formula for twisted periods
$P_\chi(\phi_n)$, which are of interest in number theory.
In particular we show (\ref{finasym}) also holds for $P_\chi(\phi_n)$ (Theorem \ref{thm3}). 
In the arithmetic case, we remark one should be able to obtain similar estimates
from subconvexity.  It is also likely that the other approaches mentioned above can
be modified to deal with the case of twisted periods.
By allowing for two different characters $\chi_1$ and $\chi_2$ in our trace formula, 
we also obtain simultaneous nonvanishing results for two
different twists $P_{\chi_1}(\phi_n)$ and $P_{\chi_2}(\phi_n)$.  Such simultaneous 
nonvanishing statements are typically quite difficult to prove.

A natural question then would be an analysis of the error term in (\ref{finasym}).
Both \cite{Hej2} and \cite{Zel2} conclude the error term is $O(1)$, while
\cite{Pitt} obtains a weaker error bound of $O(\log x)$.  To obtain this,
\cite{Zel2} uses a Tauberian theorem of Hormander which gives a sharp bound on the 
remainder.  For our setup, we are required to use a different Tauberian theorem, for which
a {\em very} crude error bound can be obtained by Tauberian remainder theory.  Of course,
with either further analysis (similar to that done for Weyl's estimate with remainder)
or a sharp Tauberian remainder theorem, one should expect the same error bound to
come out of our approach, but that is not one of our goals here.  The very 
crude bound on the error coming out of existing Tauberian remainder theory is explained
at the end of Section \ref{marksec}.

\medskip
In most of the abovementioned works that study (\ref{finasym}), a slightly more general
setup is studied where one considers two closed geodesics $C_1$ and $C_2$, and the product
of periods
\[ \int_{C_1} \phi_n \bar{\int_{C_2} \phi_n}. \]
Of course this can be treated with a relative trace formula by integrating the kernel over
$C_1 \times C_2$, however the nature of the asymptotics of the sum of these products
of periods are different than (\ref{finasym}) when $C_1 \neq C_2$.  The trace formula
in this case is essentially the same as the case of $C \times C$, except that here the
``main'' geometric term vanishes and the spectral terms are no longer positive, 
so it is not as easy to get asymptotics.  In Section \ref{pairssec}
we look briefly at this case and show there are 
infinitely many $\phi_n$ such that the periods $\int_{C_1} \phi_n$ and $\int_{C_2} \phi_n$
are simultaneously nonvanishing.

\medskip
Finally, we come to the discussion of the ortholength spectrum.  
Ortholength spectra were introduced for general hyperbolic manifolds in
\cite{Bas} (called orthogonal spectra there) as well as with additional structure 
({\em{complex ortholength spectra}}) in 3 dimensions by \cite{Mey}.  
The paper \cite{Bas} concerns a much more general setup than just two curves, 
but consider the following situation.  Let $C_1$ and $C_2$
be closed geodesics on $X$. Let $\mathcal O(X;C_1,C_2)$
denote the full orthogonal spectrum of $X$ relative to $C_1$ and $C_2$, by which we
mean the set of common orthogonals which start on $C_1$ and end
on $C_2$.   One can write $\mathcal O(X; C_1, C_2) = \bigcup_k \mathcal O_k(X; C_1,C_2)$ 
where each $\mathcal O_k(X; C_1,C_2)$ is the subset of $\mathcal O_k(X; C_1,C_2)$ comprised
of curves which cross $C_2$ exactly $k$ times.  In this context, the main theorem of
\cite{Bas} essentially states that if $C_1$ and $C_2$ are disjoint and {\em simple},
\begin{equation}\label{basthm}
  \sum_{\gamma \in \mathcal O_k(X; C_1,C_2)} \log \coth\(\f {\len(\gamma)}2\) = 2\, \len(C_1). 
\end{equation}
This is rather striking and is a sort of relative analogue of McShane's identity.
Unfortunately, the author then claims an asymptotic on the growth of 
$\mathcal O_k(X; C_1,C_2)$, but this clearly cannnot follow from just (\ref{basthm}) 
as suggested in \cite{Bas}.  Indeed such asymptotics are rather slippery as we will
see below.

In any case, we are interested in asymptotics for the full ortholength spectrum of $C$
(not necessarily simple), which is a natural consideration in light of our relative trace
formula.  However this problem is not explicitly studied in the previously mentioned works
studying (\ref{finasym}), though related problems are considered in \cite[Chapter 11]{Good}.
Define 
$\delta(\gamma) = 2\cosh(\len(\gamma))$ (which is approximately $e^{\len(\gamma)}$ 
for long $\gamma$) and set
\[ \pi_\delta(x)= \#\{\gamma \in \mathcal O(X; C,C): \delta(\gamma)<x \} = 
\sum \pi_\delta^{(k)}(x). \]
Lattice point estimates used in Section \ref{secfour} yield a very crude bound of
$\pi_\delta(x) = O(x^2)$.  To get better estimates we consider our relative trace
formula from Section \ref{secfour} with the test function $\Phi(x) = e^{-tx}$ sending
$t \to 0$ in Section \ref{orthosec}.  This yields an asymptotic roughly of the form
\[ \sum \log^2\(\f{\delta t}2\) \sim 
\f{\len(C)^2}{\mathrm{vol}(X)} \f{\pi\sqrt{2}}t 
\text{ as }t\to 0^+, \]
where the sum on the left is over $\delta=\delta(\gamma)$ for $\gamma \in \mathcal O(X;C,C)$.
This gives an upper bound of 
\[ \pi_\delta(x)  = O\( x^{1+\epsilon}\) \text{ as $x \to \infty$} \] 
and a lower bound of 
\[ \pi_\delta(x) \gg x^{1-\epsilon} \text{ as $x \to \infty$} \]
for any $\epsilon > 0$.

We remark that this situation is analogous with that described in \cite{Hej},
where one can use the heat kernel with $t\to \infty$ to obtain Weyl's law.  On the other
hand, sending $t\to 0$ yields information on the geometric terms, but
to get the full Prime Geodesic Theorem, one needs to choose another kernel.
Similarly, one expects that a different choice of kernel here should lead to a precise
asymptotic on $\pi_\delta(x)$.  It would be interesting to see what other kernels give,
but do not do this now.  Instead, in Section \ref{Goodsec} we will interpret
Good's work \cite{Good} in this context to conclude
\[ \pi_\delta(x) \sim \f{\len(C)^2}{{\pi}\text{vol}(X)} x, \]
as well as give asymptotic bounds for the growth of $\mathcal O(X;C_1,C_2)$.
The manuscript \cite{Good} is a rather thorough generalization of Kuznetsov's formula
for $G=\PSL_2(\R)$, but it is not evident what results follow from these formulas.
Hence this section may be useful for anyone trying to understand \cite{Good}.
Finally, we note that the approach in \cite{Good} is more complicated than ours in several 
ways, even if one restricts to the case of smooth compact $X$.  This, along with the problem
of obtaining an exact asymptotic for $\mathcal O(X;C_1,C_2)$, suggests a more detailed 
study of the relative trace formula we consider here may be of interest.

\medskip
\noindent
{\bf Acknowledgements.}  We would like to thank Dinakar Ramakrishnan for suggesting this project as well
as many encouraging conversations.  We would also like to thank Akshay Venkatesh for helpful  discussions
about period bounds and subconvexity, as well as Nathan Dunfield, Max Forester, 
Walter Neumann and Ara Basmajian 
for discussions related to the ortholength spectrum.
We thank the referees for helpful comments and references.
The first author was partially supported by NSF grant DMS-0402698 and a JSPS Postdoctoral Fellowship (PE07008).
\section{The relative trace formula}

 \label{secthree}

 \subsection{The setup} \label{setupsec}
Recall the identification $X = \Gamma \bs \H.$ It can be chosen so that
the preimage of $C$ in $\H$ is $i\R^+.$  For simplicity, we assume $C$ is primitive, i.e., it does not
wrap around itself.  Let $\Gamma_0$ be the diagonal
subgroup of $\Gamma$, so that $C = \Gamma_0 \bs i\R^+$.  We may write
\[ \Gamma_0 = \left< \bmx m & \\ & m^{-1} \emx \right>, \]
where $m > 1$. 

Next recall the usual setup for the Selberg trace formula on $X$,
e.g., as presented in \cite{Hej}.  Let  
\[
d(z,w) = \cosh^{-1}(1+u(z,w)/2) 
\]
 be the hyperbolic distance in the upper half plane, where 
\[ u(z,w) = \f{|z-w|^2}{\Im(z)\Im(w)}. \]
Let $\Phi$ be a smooth function on $\R_{\geq 0}$ which decays rapidly at
infinity, i.e., $\Phi(x) = O(x^{-N})$ for all $N \geq 0.$ We form
the kernel 
\begin{equation} \label{geoker}
 K(z,w) = \sum_{\gamma \in \Gamma} \Phi(u(\gamma z, w)).
 \end{equation} 
Let
\[ Q(x) = \int_x^\infty \f {\Phi(t)}{\sqrt{t-x}}dt \, (x \geq 0), \]
\[ g(u) = Q(2\cosh u - 2), \]
and
\[  h(r) = \int_{-\infty}^\infty g(u)e^{iru}du. \]
Then one also has the spectral expansion for the kernel
\begin{equation} \label{specker}
 K(z,w) = \sum_{n=0}^\infty h(r_n)\phi_n(z)\bar{\phi_n(w)}, 
 \end{equation}
where $\lambda_n = \f 14 + r_n^2$ is the eigenvalue for $\phi_n$.  Both of these expansions for
$K(x,y)$ converge absolutely and uniformly.

The relative trace formula we consider here is the identity of the geometric and spectral
expansions of
\[ \int_C \int_C K(x,y)dxdy. \]
Here, $dx$ and $dy$ denote the Poincar\'e measure. 
The spectral side of the relative trace formula, i.e., the integral of
(\ref{specker}), is evidently
\begin{equation}
\label{specside}
\int_C \int_C K(x,y)dxdy = \sum h(r_n) |P(\phi_n)|^2.
\end{equation}
The geometric side, i.e. the integral of (\ref{geoker}), is then
\[ \int_C \int_C K(x,y)dxdy = \sum_{\gamma \in \Gamma} \int_1^{m^2} \int_1^{m^2} 
\Phi(u(\gamma \cdot ix, iy)) d^\times x d^\times y. \]
Just as one breaks up the geometric side of the Selberg trace formula according to conjugacy classes,
to get the geometric side of the relative trace formula in a suitable form, we will group the summands
together by double cosets $\Gamma_0 \bs \Gamma / \Gamma_0$.  Fix a set of double coset 
representatives $\{ \gamma \}$.  Write
\[ \gamma = \bmx a & b \\ c & d \emx. \]
Observe that 
\[ \bmx m & \\ & m^{-1} \emx \bmx a & b \\ c & d \emx \bmx n & \\ & n^{-1} \emx  = \bmx mn a & mn^{-1} b
\\ m^{-1}n c & m^{-1}n^{-1}d \emx. \]
Hence each element of the double coset $\Gamma_0 \gamma \Gamma_0$ can be written uniquely
in the form $\gamma_0 \gamma \gamma_0'$ for $\gamma_0, \gamma_0' \in \Gamma_0$ unless
$bc = 0$ or $ad = 0$.   In the first case, we may assume $\gamma = 1$.  In the second, $\gamma$ is
elliptic.  However, $X = \Gamma \bs \H$ is smooth and of genus $\geq 2.$
Hence $\Gamma$ is strictly hyperbolic and does not contain any elliptic
elements. 
(This is done for simplicity.  If $\Gamma$ indeed contains an elliptic element, the contribution to the
trace formula will be the same as that coming from the identity.)  As in \cite{Jac1}, we will call 
$\gamma$ {\em regular} if $abcd \neq 0$.

Thus we may write the geometric side as
\begin{multline*}
 \sum_{\gamma_0\in \Gamma_0} \int_1^{m^2} \int_1^{m^2}  \Phi(u(\gamma_0 \cdot ix, iy)) d^\times x d^\times y \;+ \\
\sum_{\gamma \in \Gamma_0 \bs \Gamma/ \Gamma_0 - \Gamma_0 } 
\sum_{\gamma_0, \gamma_0' \in \Gamma_0} \int_1^{m^2} \int_1^{m^2} 
\Phi(u(\gamma \gamma_0' \cdot ix, \gamma_0 \cdot iy)) d^\times x d^\times y\, = \\
\int_0^{\infty} \int_1^{m^2}  \Phi(u(ix, iy)) d^\times x d^\times y \,+
\sum_{\gamma \in \Gamma_0 \bs \Gamma/ \Gamma_0 - \Gamma_0 } 
\int_0^{\infty} \int_0^{\infty}  \Phi(u(\gamma \cdot ix, iy)) d^\times x d^\times y.
\end{multline*}

Thus the geometric side of the relative trace formula is
\[ \int_C \int_C K(x,y)dxdy = \sum_{\gamma \in \Gamma_0 \bs \Gamma / \Gamma_0} I_\gamma(\Phi) \]
where
\[ I_\Id(\Phi) = \int_0^{\infty} \int_1^{m^2}  \Phi(u(ix, iy)) d^\times x d^\times y \]
and
\[ I_\gamma(\Phi) = \int_0^{\infty} \int_0^{\infty}  \Phi(u(\gamma \cdot ix, iy)) d^\times x d^\times y \]
for $\gamma$ regular.  The expressions $I_\gamma(\Phi)$ are sometimes called (relative) orbital
integrals in analogy with the Selberg trace formula case.  We now proceed to analyze these orbital
integrals.

\subsection{The main term}
For the kernel of principal interest to us here, the term $I_\Id(\Phi)$
will asymptotically dominate
the other geometric terms, and thus we will call $I_\Id(\Phi)$ the main
term.  To compute it, observe that 
\[  I_{\Id}(\Phi) = \int_0^\infty \int_1^{m^2} \Phi(u(ix,iy))d^\times x d^\times y =
  \int_0^\infty \int_1^{m^2} \Phi\(\f xy -2 + \f yx\) d^\times x d^\times y, \]
which, by the change of variables $u = \f xy$ and $v = xy$, is
\[ \int_0^\infty \int_1^{m^2} \Phi(u+u^{-1}-2) d^\times u d^\times v
= 2\log m \int_0^\infty \Phi(u+u^{-1}-2)d^\times u. \]
It will be helpful to rewrite this by symmetry as
\[ 4\log m \int_1^\infty \Phi(u+u^{-1}-2)d^\times u. \]
Making the change of variables $x=u+u^{-1}$ one finds that 
$dx = (u-u^{-1})d^\times u = (2u-x)d^\times u$.  Solving for $u$, we have $u=x+\sqrt{x^2-4}$ so
$d^\times u = (x^2-4)^{-1/2} dx$.  Hence we can rewrite the above as the integral
\begin{equation}
 I_{\Id}(\Phi) = 4\log m \int_{2}^\infty \f{\Phi(x-2)dx}{\sqrt{x^2-4}}.
\end{equation}
Observe that, whereas as the main term of the Selberg trace formula involves $\mathrm{vol}(X)$, the main term
of our relative trace formula involves $\len(C) = 2 \log m$.

\subsection{The regular terms}

A regular term ($abcd\neq 0$) on the geometric side of the trace formula is then
\[ I_\gamma(\Phi) = \int_0^\infty \int_0^\infty \Phi(u(\gamma \cdot ix, iy))
d^\times x d^\times y. \]
This depends only on the double coset representative of $\gamma$.

Write
\[ \gamma = \bmx a & b \\ c & d \emx. \]
First observe that
\begin{equation}\label{gammaix}
\gamma \cdot ix = \f{ac x^2 + bd + ix}{c^2x^2+d^2}. 
\end{equation}
Thus
\[ u(\gamma \cdot ix, iy) =
\f{(acx^2+bd)^2 + (x-y(c^2x^2+d^2))^2}{xy(c^2x^2+d^2)}. \]
Then one may check that
\[ \f{(acx^2+bd^2)+x^2}{c^2x^2+d^2} = a^2x^2+b^2. \]
Plugging this in and expanding out terms gives
\begin{eqnarray}
u(\gamma \cdot ix, iy) &=& \f{ax^2+b^2}{xy}-2+\f{y(c^2x^2+d^2)}{x} \nonumber \\
&=& \f{a^2x}{y}+\f{b^2}{xy}+c^2xy+d^2\f yx - 2.  \label{uixiy}
\end{eqnarray}
Making the substitution $z = x/y$ and $w=xy$ gives
\begin{eqnarray*}
I_\gamma(\Phi) &=& 
\f 12 \int_0^{\infty} \int_0^\infty
\Phi[a^2z+d^2z^{-1}+b^2w^{-1}+c^2w -2]
d^\times z d^\times w \\
&=& 2\int_{|b/c|}^\infty \int_{|d/a|}^\infty
\Phi[a^2z+d^2z^{-1}+b^2w^{-1}+c^2w -2]
d^\times z d^\times w,  
\end{eqnarray*}
where the inner integral is over $z$ here.  This last step follows from
the observations that the integrand and Haar measures are invariant under
the substitutions $z \mapsto d^2/(a^2z)$ and $w \mapsto b^2/(c^2w)$, and that
$z = d^2/(a^2z)$ and $w=b^2/(c^2w)$ when $z=|d/a|$ and $w=|b/c|$.  The purpose
of making the lower limit of integration non-zero is to do the following
change of variables.

Observe that by making appropriate substitutions,
\begin{eqnarray*}
 \int_r^s f(ax+bx^{-1} + c)d^\times x &=& 
\int_{\log r}^{\log s} f(ae^t + be^{-t} + c)dt \\
&=& \int_{ar + b/r}^{as+b/s}f(u+c) \f{du}{\sqrt{u^2-4ab}}.
\end{eqnarray*}
Applying this observation to change the variables $z$ and $w$ to
$u$ and $v$, we have

\begin{equation} \label{orbintadd}
I_\gamma(\Phi) =
 \int_{2|bc|}^\infty
\int_{2|ad|}^\infty
\Phi[u + v - 2]
\f{du}{\sqrt{u^2-(2ad)^2}}
\f{dv}{\sqrt{v^2-(2bc)^2}}.
\end{equation}
(The inner integral is over $u$ and the outer over $v$.)  This integral is
indeed independent of the choice of double coset representative $\gamma$ as
the quantities $ad$ and $bc$ are.

\medskip

Now one can try to separate variables.  Writing $\alpha = 2|ad|$ and $\beta =
2|bc|$ and using the substitution $t=u+v$, we get
\[ I_\gamma(\Phi) = \int_{\alpha+\beta}^\infty \Phi(t-2)dt
\int_\alpha^{t-\beta} \f{du}{\sqrt{(u^2-\alpha^2)((t-u)^2-\beta^2)}}. \]
Note that the inner integral over $u$ is a elliptic integral---precisely, it is
\begin{equation} \label{innerint}
 \f 2{\sqrt{t^2-(\alpha-\beta)^2}} K\( {\sqrt \f{t^2-(\alpha+\beta)^2}
{t^2-(\alpha-\beta)^2}} \),
\end{equation}
where $K$ denotes the complete elliptic integral of the first kind.\footnote{All non-elementary integral
formulas used in this text may be found in \cite{GR}.}  
Let us consider (\ref{innerint}) in two
cases.  Since $ad-bc = 1$, $ad > bc$.  If $ad > bc > 0$, then $\alpha - \beta = 2(ad-bc) = 2$.  If
$0 > ad > bc$, then $\alpha - \beta = 2(-ad+bc) = -2$.  Thus, in either case, i.e. if $abcd > 0$, then
$(\alpha - \beta)^2 = 4$.  On the other hand, if $ad > 0 > bc$, i.e. if $abcd < 0$, then $\alpha+\beta = 2(ad-bc) = 2$.  We also see that if $|\alpha \pm \beta| = 2$,
then $\alpha \mp \beta = 2|ad + bc|$.  

Hence writing $\delta = \delta(\gamma) = 2|ad+bc|$ gives 
\[ I_\gamma(\Phi) = 2\int_{\delta}^\infty \f{\Phi(t-2)}  {\sqrt{t^2-4}}
K\(\sqrt{\f {t^2-\delta^2}{t^2-4}}\) dt \]
when $abcd > 0$ and
\[ I_\gamma(\Phi) = 2\int_{2}^\infty \f{\Phi(t-2)}{\sqrt{t^2-\delta^2}} 
K\( \sqrt{\f{t^2-4}{t^2-\delta^2}} \)dt\]
for $abcd < 0$.
Observe that these integrals are indeed well defined as $\delta > 2$ when $abcd > 0$ and $\delta < 2$
when $abcd < 0$.  It will be seen later that there are only finitely many representatives $\gamma$
with $abcd < 0$, thus we will call these cosets and their representatives
exceptional. 

In summary, we have the following result.
\begin{prop} \label{rtfprop} (Relative trace formula)   For a function $\Phi$ as before, 
\begin{multline} \label{rtf}
2 \, \len(C) \int_2^\infty  \f{\Phi(t-2)}{\sqrt{t^2-4}} dt  
\\
 +\,
2 \sum_{\delta(\gamma)<2} \int_{2}^\infty
\f{\Phi(t-2)}{\sqrt{t^2-\delta^2}} \;  
K\( \sqrt{\f{t^2-4}{t^2-\delta^2}} \)dt 
\\
+ \, 
2 \sum_{\delta(\gamma) > 2} \int_{\delta}^\infty \f{\Phi(t-2)} 
{\sqrt{t^2-4}} \; 
K\(\sqrt{\f {t^2-\delta^2}{t^2-4}}\) dt 
\\
 = \, 
 \frac{\len (C)^{2}}{g_{X} - 1}\int\limits_{0}^{\infty} \Phi(x)
dx + 
\sum\limits_{n=1}^{\infty} h(r_n)|P(\phi_n)|^2.
\end{multline}
In the sums on the left, $\gamma$ runs over representatives for the nontrivial double cosets
$\Gamma_0 \bs \Gamma / \Gamma_0 - \Gamma_0$.
\end{prop}

We have proven all of this except the explicit calculation of the first
term on the spectral side. But this is easy, since $\lambda_{0} = 0$ and
$\phi_{0}$ is constant. We also used the Gauss-Bonnet theorem to give the
volume of $X$ explicitly. 

\subsection{Geometry of geometric terms} \label{geometry}

In Proposition \ref{rtfprop}, the integrals $I_\gamma(\Phi)$ were written solely in terms of
$\Phi$ and the quantity $\delta = 2|ad+bc|$, which we now proceed to interpret 
geometrically.

\begin{lem}\label{minimumdistance}
Let $\gamma = \bmx a & b \\ c & d \emx \in \PSL_{2}(\R),$ and assume $abcd
\neq 0.$ Then
\begin{equation*}
\text{ inf } \{ u(\gamma \cdot ix,  iy): \, x, y \in \R^{+} \}
\twocase{=}{0}{ if\ $abcd < 0,$}{\delta(\gamma)-2}{ if\ $abcd > 0.$}
\end{equation*}
Thus, for any $\gamma \in \PSL_2(\R)$,
\[ \max \{ 2, \delta(\gamma) \}=  2\cosh(\dist(\gamma \cdot i\R^+, i\R^+)).\]
\end{lem}

\begin{proof}
Let us define a function of two variables,
\begin{equation*}
h(x, y)= u(\gamma \cdot ix,  iy) =  \f{a^2x}{y}+\f{b^2}{xy}+c^2xy+d^2\f yx - 2,
\end{equation*}
by (\ref{uixiy}).
We are interested in the minimum of $h,$ so we need to compute the
gradient. One obtains
\begin{eqnarray} \label{assist}
h_{x}(x, y) = 0 & \Longrightarrow & x^2 =
\frac{a^2y^2+b^2}{c^{2} y^{2} + d^{2}}.
\end{eqnarray}
Plugging this expression for $x$ into $h_{y}(x, y) = 0$
leads to the values
\begin{equation*}
y^{4} = \frac{(b d)^{2}}{(a c)^{2}}.
\end{equation*}

Now we consider different cases. If $a b c d < 0,$ then set $y_{\min} =
\sqrt{- \frac{b d}{a c}}.$ Then $(a c y_{\min}^{2} + b d) = 0.$ Hence with
$x_{\min} = y_{\min} / (c^{2} y_{\min}^{2} + d^{2}),$ one gets $h(x_{\min},
y_{\min}) = 0.$

On the other hand, if $a b c d > 0$, then set $y_{\min} = \sqrt{\frac{b d}{a
c}}$ so that $a c y_{\min}^{2} + b d = 2 b d.$ Using (\ref{assist}) to
define $x_{\min},$ one obtains that
\begin{equation*}
h(x_{\min}, y_{\min}) = 2 \sqrt{1 + 4 a b c d} - 2 = 2 |ad + bc| - 2.
\end{equation*}
Moreover, $(x_{\min},
y_{\min})$ minimizes $h$ because as either $x$ or $y$ approach
either 0 from above or $\infty,$ $h(x, y)$ goes to $\infty.$ More
precisely, there exist constants $C, M \in \R_{>0}$ such that for all $N
\geq M$ and for all $(x, y) \notin [1/N, N]^{2},$ $h(x, y) \geq C \cdot
N.$ We omit the details.
\end{proof}

There are two kinds of regular elements $\gamma$: the exceptional $\gamma$
with $\delta < 2$ (i.e., $abcd < 0$), and non-exceptional $\gamma$ with
$\delta > 2$ (i.e., $abcd > 0$). 
Observe that $\dist(\gamma \cdot i\R^+, i\R^+)$ is the same as the distance between the image of $i\R^+$ 
and $\gamma \cdot i\R^+$ in the ``tube domain''
$\Gamma_0 \bs \H$ (so these curves in the tube domain intersect if and only if $\gamma$
is exceptional).\footnote{We use the term 
``tube domain'' here as it is suggestive of the geometry of the 
intermediary quotient $\Gamma_0 \bs \H$, and not in the sense of a Siegel domain.}

One may interpret this distance on the 
base manifold $X=M$ as follows.  Consider the set of relative homotopy classes of  
curves on $X$ whose endpoints lie on $C$.  By this, we mean two such curves are equivalent
if they are homotopic to each other by homotopies which vary the endpoints smoothly on $C$.
In each relative homotopy class, there is a unique arc of minimal length, and this length 
will be the distance between $i\R^+$ and some $\gamma \cdot i\R^+$ in the tube domain.
Moreover these minimal lengths arcs are precisely the geodesic segments which start
and end on $C$, meeting $C$ orthogonally at each endpoint.  Hence the non-exceptional
regular $\gamma$ parametrize the curves $\alpha_\gamma$ in orthonormal spectrum of $C$, and 
the quantities $\delta(\gamma)$ measure their length.  Precisely,
\begin{equation}
 \delta(\gamma) = 2\cosh(\len(\alpha_\gamma)).
\end{equation}

As for the exceptional terms, by considering the tube domain one sees that
they correspond to points of self-intersection on the closed geodesic $C.$ In other words,
exceptional double cosets exist if and only if $C$ is not simple. 
Note that any closed geodesic has at most a finite number of self-intersections, 
so that there are only finitely many exceptional terms.  This follows
from compactness of $X$, and is also a 
consequence of a lattice point counting argument in the next section. 

For $\gamma$ exceptional, $\delta(\gamma)$ determines the angle 
$\theta$ of self-intersection at the point corresponding to the double coset of $\gamma$. 
For example, the intersection is
transverse if and only if $\delta (\gamma) = 0.$  Specifically, 
from (\ref{gammaix}) we see that the line $\gamma \cdot i\R^+$ intersects $i\R^+$ when
$acx^2+bc=0$, i.e., at the point 
\[ iy = i\sqrt{\f{-ab}{cd}}. \]
Since $\gamma \cdot i\R^+$ is a Euclidean semicircle in $\H$ with center
$(ad+bc)/(2cd)$
we compute that the radial line from this center to $iy$ has slope
\[ \f{2\sqrt{-abcd}}{ad+bc} = \f{\sqrt{4-\delta^2}}{\pm \delta}. \]
Hence it follows that
\begin{equation} \label{angleformula}
\delta = |2\cos \theta|.
\end{equation}

In summary, we see that the relative trace formula encodes three pieces of geometric 
information, 
the length of $C$ coming from the main term, the self-intersection angles of $C$ coming
from the exceptional terms, and the ortholength spectrum coming from the remaining regular
terms.  Thus we may think of this trace formula as a relation between 
the ortholength spectrum and the orthogonal spectrum.

\section{The exponential kernel}
\label{secfour}

In light of the elliptic integrals in (\ref{rtf}), 
in order to compute the integrals $I_\gamma(\Phi)$ in a specific case, one
might want instead to separate the integrals in (\ref{orbintadd}).  This is possible
if we choose our kernel function to be
\[ \Phi(x) = e^{-tx}. \]
Here we take $t>0$, which makes $\Phi$ of rapid decay, and thus it is a
valid test function for our relative trace formula.  Then we get from
(\ref{orbintadd}) 
\[ I_\gamma(\Phi) = e^{2t} \int_{2|ad|}^\infty 
\f{e^{-tu}du}{\sqrt{u^2-(2ad)^2}} \int_{2|bc|}^\infty \f{e^{-tv}dv}{\sqrt{v^2-(2bc)^2}}. \]
Note that, for $a > 0$,
\[ \int_a^\infty \f {e^{-tu}}{\sqrt{u^2-a^2}} = K_0(at) \]
where $K_\nu$ denotes the $K$-Bessel function.  Hence
\[ I_{Id}(\Phi) = 2\,\len(C) \, e^{2t} K_0(2t) \]
and
\[ I_\gamma(\Phi) = e^{2t}K_0(2|ad|t)K_0(2|bc|t). \]  
Note that $ad = \f{2\pm \delta}4$ and $bc = \f{-2\pm \delta}4$, hence $2|ad|$ and $2|bc|$
equal, in some order, $\f{\delta+2}2$ and $\f{|\delta-2|}2$.

Thus the geometric side of the relative trace formula is
\[  2\len(C)e^{2t}K_0(2t) + \sum_\gamma e^{2t} K_0\(\f{\delta+2}2t\)K_0\(\f{|\delta-2|}2t\),
\]
where $\delta=\delta(\gamma)$ and $\gamma$ runs over a set of representatives for 
the nontrivial double cosets of $\Gamma$ by $\Gamma_0$.  The following asymptotic is 
standard, and may be found in \cite{GR}
along with other facts about Bessel functions we use later:
\begin{equation} \label{K0asymp}
 K_0(t) \sim \sqrt{\f \pi{2t}} e^{-t} \text{ as } t \to \infty. 
 \end{equation}
Hence as $t \to \infty$, the main term looks like
\begin{equation}\label{maintasymp}
 I_{Id}(\Phi) \asymp  \len(C) \sqrt{\f \pi t},
\end{equation} 
the exceptional terms grow like
\begin{equation}\label{excepasymp}
 I_\gamma(\Phi) \asymp \f {\pi }{t\sqrt{4-\delta^2}}.
\end{equation}
and the regular terms grow like 
\begin{equation}
 I_\gamma(\Phi) \asymp \f {\pi e^{-t(\delta - 2)}}{t\sqrt{\delta^2-4}}.
\end{equation}
 
Now consider at the spectral side.  We have
\[ Q(v) = \int_v^\infty \f{e^{-tx}}{\sqrt{x-v}}dx = e^{-tv}\int_0^\infty x^{-1/2}e^{-tx}dx = e^{-tv}\sqrt{\f \pi t}. \]
Then 
\[ h(r) =\int_{-\infty}^\infty Q(2\cosh x-2)e^{irx}dx = e^{2t}\sqrt{\f \pi t} \int_{-\infty}^{\infty} 
e^{-2t\cosh x} e^{irx} dx  = 2e^{2t}\sqrt{\f \pi t} K_{ir}(2t). \]
For $t > 1 + r^2$, we have
\[ h(r) = \f \pi t (1+O((1+r^2)t^{-1})). \]
From (\ref{specside}) the spectral side of the relative trace formula is
\[ \sum_{n=0}^\infty h(r_n) |P(\phi_n)|^2 = 2e^{2t}\sqrt{\f \pi t}
 \sum_{n=0}^\infty K_{ir_n}(2t) |P(\phi_n)|^2.\]
In summary, we have the following.

\begin{prop} (Relative trace formula --- exponential kernel)
  \begin{equation}\label{rtfexp}
   2\,\len(C)e^{2t}K_0(2t) + \sum_\gamma e^{2t} K_0\(\f{\delta+2}2t\)K_0\(\f{|\delta-2|}2t\)
=  2e^{2t}\sqrt{\f \pi t}\sum_{n=0}^\infty K_{ir_n}(2t)|P(\phi_n)|^2. 
 \end{equation}
 \end{prop}
 
We remark that the arguments $ir_n$ of the Bessel functions $K_{ir_n}(2t)$ appearing on the 
right are real for exceptional eigenvalues $\lambda_n < \f 14$ and purely imaginary for
$\lambda_n > \f 14$.
 
\begin{prop} \label{firstasymp} As $t \to \infty$, we have the asymptotic
\begin{equation}
2e^{2t}\sqrt{\f \pi t}\sum_{n=0}^\infty K_{ir_n}(2t)|P(\phi_n)|^2 = 
\len(C) \sqrt{\f \pi t} + \(\sum_{\delta < 2} \f 1{\sqrt{4-\delta^2}}\) \f \pi{t}
+ O\(\f 1{t\sqrt t}\),
\end{equation}
or, in a less refined form,
\begin{equation}  \label{eqperiod}
  \lim_{t \to \infty} e^{t} \sum_{n=1}^\infty K_{ir_n}(t)|P(\phi_n)|^2 = \f{\len(C)}2. 
\end{equation}
\end{prop}

\begin{proof}
We first consider the non-exceptional regular geometric terms $I_\gamma(\Phi)$.  
For $t$ large, we have the estimate
\begin{equation} \label{kbesserr}
  K_0(t) \leq \sqrt{\f \pi{2t}} e^{-t}\( 1 +  \f 1{8t} \).
\end{equation}
Hence for $t$ large, we may make the estimate 
\begin{equation*} 
I_\gamma(\Phi) = e^{2t} K_0\(\f{\delta+2}2t\)K_0\(\f{|\delta-2|}2t\)
 \leq \f {\pi e^{-t(\delta - 2)}}{t\sqrt{\delta^2-4}}
 \( 1 + \f{\delta}{2(\delta^2-4)t} + \f 1{16(\delta^2-4)t^2} \).
\end{equation*}
To estimate $\sum_\gamma I_\gamma(\Phi)$ we will need to know some bound on the growth
of $\delta$.  We can estimate a count
\[ \pi_\delta(x) = \#\{ \gamma \in \Gamma_0 \bs \Gamma / \Gamma_0 -\Gamma_0: \delta(\gamma) < x \} \]
from the lattice point problem.  The principal result we will use is that
\[ \#\{ \bmx a&b \\ c&d\emx \in \Gamma : a^2 + b^2 + c^2 +d^2 < x \} = O(x) \]
(e.g., Theorem 12.1 of \cite{Iwa}.)
We can count our set as
\[ \pi_\delta(x) =  \# \{ \bmx a&b \\ c&d \emx \in \Gamma - \Gamma_0 : 1\leq |a| < m,\, 1\leq |b| < m^2,\, 2|ad+bc| < x \}. \]
Elements in this set satisfy 
\[ c^2+d^2 +2abcd \leq (ad+bc)^2 \leq \f {x^2} 4. \]
If $abcd > 0$ we clearly have $c^2+ d^2 \leq \f{x^2}4$.  Suppose $abcd < 0$.  Then $ad > 0 > bc$ but
since $ad = 1+bc$, $0 > bc > -1$ and so $abcd > -1$.  Hence in either case we have
\[ c^2+d^2 \leq  \f {x^2} 4 + 1, \]
so
\[ \pi_\delta(x) < \#\{ \bmx a&b \\ c&d\emx \in \Gamma  : a^2 + b^2 + c^2 +d^2 < 2m^4 + \f{x^2}4 + 1 \}
= O(x^2). \]
Let $\{\delta_n\}$ denote the sequence of $\delta(\gamma)$'s in increasing order.
Then for any $\epsilon > 0$, $\pi_\delta(x) = O(x^2)$ implies
\begin{equation}\label{deltan}
\delta_n \gg n^{\f 1{2 + \epsilon}}. 
\end{equation}
We remark that this also implies
there are only finitely many $\delta < 2$, i.e., only finitely many
exceptional double cosets with $abcd < 0$, as noted in Section \ref{geometry}.

Now we may uniformly bound
\[ I_\gamma(\Phi) \leq C  \f {\pi e^{-t(\delta - 2)}}{t\sqrt{\delta^2-4}} \]
for some $C > 0$ and $t$ large.  Also, since the sum of any finite number of 
$I_\gamma(\Phi)$ goes to 0 exponentially fast as $t \to \infty$, it suffices to bound the
growth of some tail
\[ \sum_{\delta > N} I_\gamma(\Phi)\leq C \sum_{\delta > N} 
\f {\pi e^{-t(\delta - 2)}}{t\sqrt{\delta^2-4}} \leq C \sum_{n > N'}
\f {\pi e^{-t n^{1/3}}}{t n^{1/3}}. \]
Here we have used (\ref{deltan}) with $\epsilon=1$.  We can easily estimate this sum
on the right by
\[ \sum_{n > N} \f {\pi e^{-t n^{1/3}}}{t n^{1/3}} \leq 
\pi \int_1^\infty \f{e^{-tx^{1/3}}}{x^{1/3}}dx = 3\pi\int_1^\infty ue^{-tu}du
=3\pi e^{-t}\( t^{-1} + t^{-2} \). \]
Hence the contribution from the non-exceptional regular terms is
\[ \sum_{\delta > 2} I_\gamma(\Phi) = O\(\f{e^{-(\delta_0-2) t}}{t}\), \]
where $\delta_0$ is the minimum $\delta > 2$, which will be absorbed in 
the error terms below.

On the other hand, we may estimate any of the finitely many exceptional terms
by (\ref{excepasymp}), and observe that (\ref{kbesserr}) gives $O(t^{-2})$ for the
error term.  Similarly there is an $O(t^{-3/2})$ error term coming from the
main term estimate (\ref{maintasymp}).  
Putting these estimates in the relative trace formula 
yields the first assertion of Proposition \ref{firstasymp},
and the second follows from identifying dominant terms.
\end{proof}

Combining the asymptotics (\ref{K0asymp}) and (\ref{eqperiod}) yields the following

\begin{cor}  $P(\phi_n) \neq 0$ for infinitely many $n$.
\end{cor}

\section{Refined estimates} \label{marksec}

In this section, we will use $z$ for a positive variable.  We will need
uniform asymptotics of the $K$-Bessel function as well as many other
estimates to apply a Tauberian theorem to conclude the main result of this section,
Theorem \ref{thm2} below.

\subsection{Uniform $K$-Bessel estimates}

In what follows, we will be interested in asymptotics and estimates for
\begin{equation}
K_{ir}(z) = \frac{1}{2} \int_{\mathbb{R}} e^{-z\cosh{t}}  
e^{irt} dt,  \label{eqstart}
\end{equation}
as both of the parameters $r$ and $z$ vary.

In the above, $r$ and $z$ are both real, with $z>0$.  Notice, under the
transformation $t \leftrightarrow -t$, the integral above is
automatically real.  

We first will be interested in the case  $4 \le r \le z $.

Let us make the change of variables
\begin{equation}
\cosh{t} - 1 = s^2~;~~~~s\cdot t \ge 0.   \label{eqchoshts1}
\end{equation}
The above restriction forces $s<0$ when $t<0$, and similarly
$s>0$ when $t>0$. One can see this change of variables is a
diffeomorphism from $\mathbb{R}\leftrightarrow \mathbb{R}$, and we can
write  $t$ as a function of $s$ as ~  $t(s) = \cosh^{-1}(s^2+1)$, for
$s>0$ by equation (\ref{eqchoshts1}).  With no restriction on $s$,
explicitly,~
 $t(s)= \ln(1+s^2 + s\sqrt{s^2+2})$.
 The differential can easily be computed:
\begin{equation}
\frac{dt}{ds} = \frac{2}{\sqrt{s^2 + 2}}~~~~s\in \mathbb{R}.  
\label{eqchoshts2}
\end{equation} 

It follows that
\[
K_{ir}(z) = \frac{e^{-z} }{2} \int_{\mathbb{R}} e^{-zs^2} e^{irt(s)} 
\frac{2}{\sqrt{s^2 + 2}} ~ds.
\]

\begin{lem}
For $r>0$ and $z>0$, we have
\[
K_{ir}(z) = \frac{e^{-z} }{2}
\int_{-\frac{1}{\sqrt{r}}}^{\frac{1}{\sqrt{r}}}
 e^{-zs^2} e^{irt(s)} 
\frac{2}{\sqrt{s^2 + 2}} ~ds  +\frac{e^{-z}}{\sqrt{z}}
O\left( \frac{\sqrt{r}}{\sqrt{z}}e^{-\frac{z}{r }}\right).
\]   \label{lemma1}
\end{lem}

Notice that to really use this, $r$ must be smaller than $z$
by a power of $z$;  the region $4 \le r \le z^{14/25}$ will be
of interest later.  The error constant here is uniform.

\begin{proof} In light of the expression following equation (\ref{eqchoshts2}),
we need only show
\[ e^{-z}
\int_{(-\infty,-\frac{1}{\sqrt{r}}]\cup [\frac{1}{\sqrt{r}}, \infty)}
 e^{-zs^2} e^{irt(s)} 
\frac{2}{\sqrt{s^2 + 2}} ~ds 
\] 
can be absorbed into the error term. This is trivially bounded 
by
\[
4e^{-z} \int_{\frac{1}{\sqrt{r}}}^\infty e^{-zs^2} ds  = 
\frac{4e^{-z}}{\sqrt{z}} \int_{\frac{\sqrt{z}}{\sqrt{r}}}^\infty 
e^{-u^2}  ~du   <   \frac{4e^{-z}}{\sqrt{z}}  
\int_{\frac{\sqrt{z}}{\sqrt{r}}}^\infty   e^{-\sqrt{\frac{z}{r}}u}~du.
\]
The equality above is from a substitution, and the inequality uses the
region of integration.  We have $u^2 \ge \sqrt{\frac{z}{r}}u$ with 
$u \in [\frac{\sqrt{z}}{\sqrt{r}},\infty)$ (we will use a similar
estimate in Lemma \ref{lemma3}).
This last integral is trivially evaluated to be
$\frac{\sqrt{r}}{{\sqrt{z}}}e^{-\frac{z}{r}}$ and the lemma follows.
\end{proof}

We must now estimate $t(s)$ which appears inside the $s$-integral, in
Lemma \ref{lemma1}.

By (\ref{eqchoshts2}), we can write the Maclaurin series for 
(respectively) $\frac{dt}{ds}$ and $t$ as

\begin{equation}
\{
\begin{matrix}   
\frac{dt}{ds} = \sqrt{2}[1-\frac{1}{4}s^2 + \frac{3}{32}s^4+ \cdots]\\
t(s) = \sqrt{2}[s - \frac{1}{12}s^3 + \frac{3}{160}s^5 - \cdots].
\end{matrix}  \right.    \label{eqmacseries1}
\end{equation}
Convergence is uniform for both series, for $|s| < \sqrt{2}$.  In what
follows, we will be taking $|s| \le \frac{1}{\sqrt{r}}$, as well as
assuming $r \ge 4$. With these assumptions, we easily have
\begin{equation}
\{
\begin{matrix}   
\frac{dt}{ds} =  \sqrt{2}[1-\frac{1}{4}s^2] +O(s^4) \\
t(s) - \sqrt{2} s =  -\frac{\sqrt{2}}{12}s^3 +O(s^5) = O(s^3),
\end{matrix}  \right.     \label{eqmacseries2}
\end{equation}
with uniform constants in all error terms, since
$s \in [-1/\sqrt{r},1/\sqrt{r}]\subset [-1/2,1/2]$.

Now we can \textit{replace} the $e^{irt(s)}$ term in equation
(\ref{eqstart}) with $\cos(rt(s))$.
From a simple calculus theorem, we have
\[
\cos(rt(s)) = \cos(\sqrt{2}rs) - r(t(s)-\sqrt{2}s) \cdot \sin(rc(s))
\] 
where $c(s)$ is a point between $t(s)$ and $s$. With our assumptions on
$r$ and $s$, this gives
\begin{equation}
\cos(rt(s)) = \cos(\sqrt{2}rs) +O(rs^3)   \label{eqmvthm}
\end{equation}
using equation (\ref{eqmacseries2}) along with the trivial estimate
$|\sin(rc(s))|\le 1$, again
with uniform constant.

Thus, the integral term in Lemma \ref{lemma1} we can now write as
\begin{multline}
\frac{e^{-z} }{\sqrt{2}}
\int_{-\frac{1}{\sqrt{r}}}^{\frac{1}{\sqrt{r}}}
 e^{-zs^2}  \left( \cos(\sqrt{2}irs) + O(rs^3) \right) \left(
  [1-\frac{1}{4}s^2] +O(s^4)  \right) ~ds   \\
= \frac{e^{-z} }{\sqrt{2}}
\int_{-\frac{1}{\sqrt{r}}}^{\frac{1}{\sqrt{r}}}
  e^{-zs^2}\left( \cos(\sqrt{2}irs) [1-\frac{1}{4}s^2] +
O(rs^3 +s^4)  \right)  ~ds.   \label{eq14estimates}
\end{multline}
Here the constants in the error term on the right are uniform, but
also depend on the constants in the O terms on the left coming
from equations (\ref{eqmacseries2}) and (\ref{eqmvthm}).

Essentially, after noticing the cosine term in the integral
on the right hand side of equation (\ref{eq14estimates}) can be replaced
with an exponential, these
computations give 
\begin{lem}  For $4 \le r \le z$ we have
\[
K_{ir}(z) =   \frac{e^{-z} }{\sqrt{2}}
\int_{-\frac{1}{\sqrt{r}}}^{\frac{1}{\sqrt{r}}}
 e^{-zs^2} e^{\sqrt{2}irs}    [1-\frac{1}{4}s^2] ~ds~ + ~
\frac{e^{-z}}{\sqrt{z}} O\left( \frac{r}{z^{3/2}}   +
\frac{\sqrt{r}}{\sqrt{z}}e^{-\frac{z}{r}} \right).
\]   \label{lemma2}
\end{lem}

\begin{proof} Following the Taylor estimates on $t(s)$ and $\frac{dt}{ds}$
after Lemma \ref{lemma1} (as well as using Lemma \ref{lemma1})
all the way to equation (\ref{eq14estimates}),
we see the main term above is the integral on the right side of
(\ref{eq14estimates}).  Thus, all we need to show is that the error terms
from the right side of (\ref{eq14estimates}) account for the first  
term  inside the error term above.

Easily, one error term from the right side of
(\ref{eq14estimates}) is
\[ e^{-z} \cdot O \left( 
\int_{-\frac{1}{\sqrt{r}}}^{\frac{1}{\sqrt{r}}}
 |e^{-zs^2} e^{\sqrt{2}irs} r s^3| ~ds    
\right). \]
This can trivially be bounded by
\begin{eqnarray*}
re^{-z}  O \left( 
\int_{-\frac{1}{\sqrt{r}}}^{\frac{1}{\sqrt{r}}}
e^{-zs^2} |s|^3 ~ds
  \right) &=& \frac{ re^{-z}}{\sqrt{z}}  O \left( 
\int_0^{\sqrt{z/r}} e^{-u^2} \left( \frac{u}{\sqrt{z}} \right)^3 ~du
\right)\\
&=& \frac{e^{-z}}{\sqrt{z}}  O \left(  \frac{r}{z^{3/2}} \int_0^\infty
e^{-u^2} u^3 ~du \right).
\end{eqnarray*}
The first equality is a substitution.  The second is extending the
integral to half the real line (so the interval is independent of
$r$ or $z$).  This last integral converges, but we will have to
incorporate its value into the error term of this lemma.  

The other error term
produces $\frac{e^{-z}}{\sqrt{z}}  O \left(  \frac{1}{z^2} \right)$,
which we simply absorb into the  
$\frac{e^{-z}}{\sqrt{z}} O\left( \frac{r}{z^{3/2}}  \right)$ term by our
assumption $4\le r \le z$.
\end{proof}

This leads us naturally to
\begin{lem}  \label{lemma3}
  For $4 \le r \le z$ we have
\[
K_{ir}(z) =   \frac{e^{-z} }{\sqrt{2}}
\int_{\mathbb{R}}
 e^{-zs^2} e^{\sqrt{2}irs}    [1-\frac{1}{4}s^2] ~ds~ + ~
\frac{e^{-z}}{\sqrt{z}} O\left( \frac{r}{z\sqrt{z}}
+\frac{\sqrt{r}}{\sqrt{z}} e^{-\frac{z}{r}}  \right).
\]   
\end{lem}

Note; the constant here is uniform, but possibly different than
previous lemmas.

\begin{proof}  By Lemma \ref{lemma2}, we need to estimate
\[
e^{-z}\int_{\frac{1}{\sqrt{r}}}^\infty 
| e^{-zs^2} e^{\sqrt{2}irs} [1-\frac{1}{4}s^2]|~ds     <
e^{-z}\int_{\frac{1}{\sqrt{r}}}^\infty e^{-zs^2} 
[1+\frac{1}{4}s^2]~ds  .
\]
By a change of variable, the right side here is now equal to
\[
\frac{e^{-z}}{\sqrt{z}}\int_{\sqrt{\frac{z}{r}}}^\infty 
e^{-u^2}[1+\frac{u^2}{4z}]~du.
\]
By estimates very similar to Lemma \ref{lemma1}, as well as an
integration by parts, this can be shown to be of the size 
~$\frac{e^{-z}}{\sqrt{z}} O (\frac{\sqrt{r}}{\sqrt{z}}e^{-z/r})$,
where the uniform constant is actually different than that of Lemma
\ref{lemma1}.  
The only error term in this lemma unaccounted for  comes
from the error term  from Lemma \ref{lemma2}. 
\end{proof}

We will evaluate the integral in Lemma \ref{lemma3} by using
Fourier transforms.

Suppose $f(x)$ is a Schwartz function on $\mathbb{R}$.  We define
the Fourier transform of $f$ as
\[
\hat{f}(y) = \int_\mathbb{R} ~f(x) e^{ixy} ~dx.
\]
Since $f$ is a Schwartz function, $\hat f$ also is. Furthermore,
\[
\frac{d\hat{f}}{dy} (y) = \int_\mathbb{R} ~f(x) \cdot (ix) e^{ixy}  ~dx
~\implies  ~\frac{d^2\hat{f}}{d^2y} (y) = -\int_\mathbb{R} ~f(x)
\cdot x^2 e^{ixy}   ~dx.
\]

Note that with our normalization the Fourier transform of the
Gaussian function $f(x) = e^{-x^2}$ is
\begin{equation}
\hat{f}(y) = \int_\mathbb{R} ~e^{-x^2} e^{ixy} ~dx = \sqrt{\pi}
e^{\frac{-y^2}{4}}.  \label{eqftransfrom}
\end{equation}
This brings us easily to

\begin{prop}
For $4 \le r \le z$, we have
\[
K_{ir}(z) = \frac{\sqrt{\pi}}{\sqrt{2z}}e^{- (z +\frac{r^2}{2z})} \left[
1 + \frac{r^2-z}{8z^2}
\right]~ +~ \frac{e^{-z}}{\sqrt{z}} O\left( \frac{r}{z\sqrt{z}}
+\frac{\sqrt{r}}{\sqrt{z}} e^{-\frac{z}{r}}  \right).
\]   \label{prop1}
\end{prop}

\begin{proof} We use Lemma \ref{lemma3} along with the above Fourier theory
to compute the integral of Lemma \ref{lemma3} \textit{exactly}. 
Consequently, the error term here is also exactly that of Lemma 
\ref{lemma3}.

By the change of variable~ $x = \sqrt{z}s$, the integral in Lemma
\ref{lemma3} becomes
\[
\frac{e^{-z}}{\sqrt{2z}} \int_{\mathbb{R}}
e^{-x^2} e^{ix(\frac{\sqrt{2}r}{\sqrt{z}})} 
\left(1 - \frac{1}{4z}x^2 \right)  ~dx.
\]
So, one term here is exactly the Fourier transform of $e^{-x^2}$, 
evaluated at  $y = \frac{\sqrt{2}r}{\sqrt{z}}$.
Applying equation (\ref{eqftransfrom}) to  this first term,
as well as the term
involving the $x^2$ (which brings a second $y$-derivative of
$\sqrt{\pi}
e^{\frac{-y^2}{4}}$)
gives our result.  
\end{proof}

(Note: If we fix $r$ and send $z$ to $\infty$ we recover the asymptotic
$\sqrt{\frac{\pi}{2z}}e^{-z}$, which is valid for any fixed $r$ by
Laplace's method. 
This proposition gives us an asymptotic
for $z$, but is also \textit{uniform}
 for $r$ of size even much larger
than $\sqrt{z}$. For example, one will see that the $r$ values in the
region  ~$\sqrt{z} \le r \le \sqrt{z\log(z)}$~ will still contribute to
the sum (\ref{eqperiod}).)

\begin{prop}
For $z$ sufficiently large, and $z^{\frac{14}{25}} \le r \le z$, we have
\[
K_{ir}(z) = ~e^{-z} \cdot O \left( \frac{z^{\frac{15}{2}}}{r^{16}} 
\right)  .
\]  \label{prop2}
\end{prop} 

\begin{proof}  With our assumptions on $r$ and $z$, integrate by-parts
\[
2K_{ir}(z) =  \int_{\mathbb{R}} e^{-z\cosh{t}}  
e^{irt} dt
\]
15 times.
Note that $z$ must be very large here.   

Specifically, as a smaller example we will show
$K_{ir}(z) =  e^{-z}O(\frac{z^{5/2}}{r^6})$ after 5 integrations
by-parts. We have
\begin{equation}
  2K_{ir}(z) = \frac{-1}{ir^5} \int_{\R} \left[
 \frac{d^5}{dt^5} e^{-z\cosh{t}}   \right]
   e^{irt} dt  ,   \label{eq5byparts}
\end{equation}
where 
\begin{equation}
\frac{d^5}{dt^5} e^{-z\cosh{t}} 
=\{ \begin{matrix}
 -z^5\sinh^5(t) + 10
z^4\sinh^3(t)\cosh(t) - 25z^3\sinh^3(t)\\
~~~~~~~~~ + 15z^2\sinh(t)\cosh(t) -
z(15z^2+1)\sinh(t)  \label{eq5diffK}
\end{matrix}\} \cdot e^{-z\cosh{t}} .
\end{equation}
Our strategy for obtaining an
estimate at 15 integrations is the same for 5, which we explain now.

Now, we replace (\ref{eq5diffK}), which has 5 terms, into the integral of 
(\ref{eq5byparts}), and then separate to make 5 integrals.  We obtain
an estimate for the size of each integral.
For example, if we take the second term on the right
side of (\ref{eq5diffK}), we become interested in the 
size of the integral
\begin{equation}
 \frac{z^4}{r^5} 
 \int_{\mathbb{R}} \left[ \sinh^3(t)\cosh(t) e^{-z\cosh{t}}   \right]
   e^{irt} dt  .    \label{eqsize2ndterm}
\end{equation}

Let us take a closer look at the function $\sinh^3(t)\cosh(t) e^{-z\cosh{t}} $
appearing in (\ref{eqsize2ndterm}).  
For convenience, in this proposition, let us use
$h_z(t)=  \sinh^3(t)\cosh(t) e^{-z\cosh{t}}$.
For each $z$, $h_z(t)$ is odd. Its derivative is
\begin{equation}
 e^{-z\cosh{t}}  \sinh^2(t) [ -z\sinh^2(t)\cosh(t)+ 4\sinh^2(t) + 3],
     \label{eqderivalt}
\end{equation}
which is clearly zero at $t=0$.  The term in brackets has only one
zero for $t>0$ (recall, $z$ is very large).  This happens when
\[
z = \frac{4}{\cosh(t)} + \frac{3}{\sinh^2(t)\cosh(t)}.
\]
Using the Maclaurin series for both $\sinh(t)$ and $\cosh(t)$,we find this
zero happens at a value of $t$ that is of order $1/\sqrt{z}$.  In other
words, for $z$ sufficiently large, we have the other zero (for $t>0$)
of the derivative  (\ref{eqderivalt}) occurs at some $t_0(z)$ with
\begin{equation}
 t_0(z) \asymp \frac{\sqrt{3}}{\sqrt{z}}.  \label{eqtosqrtz}
\end{equation}
Consequently, ~$h_z(t)=\sinh^3(t)\cosh(t) e^{-z\cosh{t}}$
is strictly increasing (resp. decreasing) on $(0,t_0(z))$ (resp.
($t_0(z),\infty)$).  Now, since this function is odd,
we are integrating    ~$h_z(t)$     against 
$\sin(rt)$ in (\ref{eqsize2ndterm}), which has period $2\pi/r$.  We will
break up each interval where ~$h_z(t)=\sinh^3(t)\cosh(t) e^{-z\cosh{t}}$
is monotone into segments of length
$ \pi/r$ with endpoints $ k\pi/r, k\in {\mathbb{Z} }$.  
More specifically, put $n_1(z)\in \mathbb{N}$ be the first 
(smallest)  integer so that
$ n_1(z)\pi/r \ge t_0(z)$. Then
\[
\int_{t_0(z)}^\infty \sinh^3(t)\cosh(t) e^{-z\cosh{t}}\sin(rt)~dt 
\]
\[
=\int_{t_0(z)}^{ n_1(z)\pi/r} h_z(t)\sin(rt)~dt
 +  \sum_{k = n_1(z)}^\infty
\int_{ k\pi/r}^{ (k+1)\pi/r} h_z(t) \sin(rt)~dt
 \]
\[ =
\int_{t_0(z)}^{ n_1(z)\pi/r} h_z(t) \sin(rt)~dt
 +\sum_{k = n_1(z)}^\infty a_k,
\]
with the obvious definition for $a_k$.  Since 
$h_z(t)= \sinh^3(t)\cosh(t) e^{-z\cosh{t}}$ is strictly decreasing on 
$[t_0(z), \infty)$, we see that the series $ \sum_{k = n_1(z)}^\infty a_k$
is an alternating series where $|a_{k+1}| \le |a_k|$.  Consequently, the
sum of this series is in absolute value $\le |a_{n_1(z)}|$.

How big can $|a_{n_1(z)}|$ be? We are integrating over an interval of
length $\pi/r$. Further, the maximum of the integrand must be, using
(\ref{eqtosqrtz}), at most
\[
 \left. \sinh^3(t)\cosh(t) e^{-z\cosh{t}}\right|_{t_0(z)} =
O\left( \frac{e^{-z}}{z^{3/2}} \right).
\]
This gives us a total of $O\left( \frac{e^{-z}}{rz^{3/2}} \right)$.
The integral 
\[
\int_{t_0(z)}^{ n_1(z)\pi/r}\sinh^3(t)\cosh(t)
e^{-z\cosh{t}}\sin(rt)~dt
\] is handled similarly (the interval is shorter than $\pi/r$).

This leaves us with 
\[
\int_0^{t_0(z)} \sinh^3(t)\cosh(t)
e^{-z\cosh{t}}\sin(rt)~dt.
\]
By our assumption on $r$ and $z$, namely that
$z^{14/25} \le r$, this implies that $2\pi/r$ (the period of the
oscillating factor) is \textit{much} less than
$t_0(z)$.  Consequently, we can do the same type of analysis for the
interval $(0,t_0(z))$, and actually conclude the same result.
 Here, we will have a finite alternating sequence, whose absolute
value terms are increasing.  This means we use the last term for
estimates; this brings us to an interval
$[(n_1(z)-1)\pi/r, t_0(z)]$, and once again we use that
$h_z(t)=\sinh^3(t)\cosh(t)e^{-z\cosh{t}}\sin(rt)$
has a maximum at $t_0(z)$.  For
$t<0$, we just use symmetry; ~$\sinh^3(t)\cosh(t) e^{-z\cosh{t}}\sin(rt)$
is an even function of $t$ for any $z$.  We conclude the same result, and
leave the details to the reader.
  
This gives us the contribution of the 
~$h_z(t)=\sinh^3(t)\cosh(t) e^{-z\cosh{t}}$ term from (\ref{eq5diffK})
into (\ref{eq5byparts}) is  $e^{-z}O(\frac{z^{5/2}}{r^6})$.
  This bound
holds for \textit{each} of the 5 pieces  (\ref{eq5diffK}).
Consequently by
(\ref{eq5byparts}), we have $K_{ir}(z) = e^{-z}O(\frac{z^{5/2}}{r^6})$,
with our assumptions on $r$ and $z$.  We see from computation that 
\textit{each} integration by parts gives us an extra factor of 
$\sqrt{z}/r$, keeping our assumptions on $r$ and $z$.
Integrating by parts 10 more times, this gives the proposition, since we
have a factor of $\sqrt{z}/r$ each time.
\end{proof}

\medskip
We remark that if one restricts to special regions, one has better estimates 
for the $K$-Bessel functions than those in Propositions \ref{prop1} and \ref{prop2},
but the contributions from these regions are negligible for our asymptotics.
 
\subsection{An eigenvalue moment estimate}
In this subsection, we record estimates needed for our main theorem below.

\begin{lem}
Let $\alpha \in \mathbb{R}$ with $\alpha \not= -2$. Let $m_1 , m_2 >0$
with $m_2 > m_1$.  Then  \[
\lim_{z \rightarrow \infty} \sum_{z^{m_1} \le r_j \le z^{m_2}} 
r_j^\alpha  = O \(z^{(\alpha  +2)m_1}+z^{(\alpha  +2)m_2} \).
\]  Further, if $\alpha > -2$, \[
\lim_{z \rightarrow \infty} \sum_{4 \le r_j \le z^{m_1}}  r_j^\alpha 
= O \( z^{(\alpha  +2)m_1} \).
\]   \label{lemma4}
\end{lem}

\begin{proof}  First, let us define the function $N(x)$ to be the total
number of eigenvalues $\lambda_j$ with $\lambda_j \le x$.  Clearly, we are
taking $x \ge 0$, and we are counting eigenvalues with multiplicity.
For our situation, it is well-known  that Weyl's
law holds (\cite{Hej}, \cite{Iwa}), i.e., 
\begin{equation}
N(x) ~\sim~ \frac{\text{vol}(\Gamma \backslash \mathbb{H})}{4\pi} \cdot x
~~~\text{as}~x \rightarrow \infty.  \label{eqWeyl}
\end{equation}

Further, let us note that since $\lambda_j = \frac{1}{4}+r_j^2$, we have
\begin{equation}
r_j ~\asymp ~\sqrt{\lambda_j}   \label{eqrj}
\end{equation}
with uniform constants, since we will always be considering only
$r_j \ge 4$.  
By (\ref{eqWeyl}), we also have
\begin{equation}
N(x)  \le  c_1 x = O(x)~~~~~~~x\ge 1 ,  \label{eqnox}
\end{equation}
for a universal constant $c_1$. We will only need to
consider $\lambda_j \ge \f{65}4$ in the future for estimates.

Let us consider 
\[\lim_{z \rightarrow \infty} \sum_{z^{m_1} \le r_j \le z^{m_2}} 
r_j^\alpha.\]
  Since  $m_1>0$, and we are taking $z$ to $\infty$, we may
assume $z^{m_1} > 4$, so that the statements in the above paragraph
hold.  For $z$ large, by (\ref{eqrj}) notice that
\[
 \sum_{z^{m_1} \le r_j \le z^{m_2}} r_j^\alpha  ~\asymp
 \sum_{z^{m_1} \le r_j \le z^{m_2}} \lambda_j^{\frac{\alpha}{2}} ,
\]with universal constants depending on $m_1$, $m_2$, and $\alpha$, but
not on $z$.  Now
\[
 \sum_{z^{m_1} \le r_j \le z^{m_2}} \lambda_j^{\frac{\alpha}{2}} = 
\int_{z^{2m_1}+1/4}^{z^{2m_2}+1/4}   x^{\alpha/2}~dN(x).
\]By partial integration, this equals
\[
 \left.  x^{\alpha/2} N(x) \right|_{z^{2m_1}+1/4}^{z^{2m_2}+1/4}
-\frac{\alpha}{2} \int_{z^{2m_1}+1/4}^{z^{2m_2}+1/4}
x^{\frac{\alpha}{2}-1} N(x) ~dx  .
\]
Using Equation (\ref{eqnox}) for \textit{both} terms here, along with
simple estimates, gives the first part of this lemma, where the $O$
constant is allowed to depend on $\alpha$.  If
$\alpha > -2$, the $m_2$ term will dominate, while if $\alpha < -2$, the
$m_1$ term will dominate.  The second part of this lemma is proved the
same way.
\end{proof}

\subsection{The main result}

Using the estimates we have developed, we can now prove
\begin{thm} With notation as above,
\begin{equation} \label{thm1laplace}
\lim_{z \rightarrow \infty} ~~\frac{\sqrt{\pi}}{\sqrt{2z}}
\sum_{j=0}^\infty  e^{-\frac{r_j^2}{2z}}  \cdot
\left| \int_C  \phi_j \right| ^2 = \frac{1}{2} \len(C).
\end{equation}  \label{thm1}
\end{thm}

\begin{proof}  Once again, consider the  sum in Equation (\ref{eqperiod}), i.e.,
\begin{equation}
e^{z} \sum_j  K_{i r_j}(z) \left| \int_C \phi_j \right| ^2\label{eqbasic}
\end{equation}

There are a finite number of
terms in the sum here, with
$r_j < 4$.  For \textit{each} of these terms, one has a limit of 
$\sqrt{\frac{\pi}{2z}}e^{-z}$ as $z \rightarrow \infty$ from Laplace's
method.
A result of Reznikov \cite{Reznikov}
gives
\[
 \int_C  |\phi_j|^2  = O(\lambda_j^{1/4}),
 \]  
and so by Cauchy-Schwartz, 
 \begin{equation}
\left| \int_C  \phi_j \right| ^2   = O(\lambda_j^{1/4}).   \label{eqrez}
\end{equation}
This bound will be important for us.  Note that the trivial $L^{\infty}$
bound, $O(\lambda_j^{1/2})$, is not sufficient for our purposes.
(In fact, Zelditch \cite[Corollary 3.3]{Zel2} showed that the periods 
$|\int_C \phi_j|$ are bounded by a constant, but his argument
essentially makes use of what we are 
trying to prove.)

In particular, the sum of the finite number of terms in (\ref{eqbasic}) 
for which $r_j <4$ can be
bounded by
\[
e^{z} \cdot  O \(\sum_{r_j < 4} \frac{e^{-z}}{\sqrt{z}}
\sqrt[4]{\f{65}4}\) 
=O\(\f {N\(65/4\)} {\sqrt{z}}\) 
\rightarrow 0 ~~\text{as}~z
\rightarrow
\infty.
\]

Consequently, we will
 separate the sum in equation (\ref{eqbasic})
into four separate regions.  These regions will be:
$4 \le r \le \sqrt{z}$, ~$\sqrt{z} \le r \le z^{14/25}$, 
~$z^{14/25} \le r \le z$, ~and $r >z$.  We will prove each bound, in a 
separate proposition, one for each region.  

\begin{prop}
\[
\lim_{z \rightarrow \infty} ~e^{z} 
\sum_{4 \le r_j \le \sqrt{z}}  K_{i r_j}(z) \left| \int_C
\phi_j \right|^2   =
\lim_{z \rightarrow \infty} ~~\frac{\sqrt{\pi}}{\sqrt{2z}}
\sum_{4 \le r_j \le \sqrt{z}} e^{-\frac{r_j^2}{2z}}  \cdot
\left| \int_C  \phi_j \right| ^2.
\]  \label{prop3}
\end{prop}

\begin{proof} In light of the error terms from Proposition \ref{prop1} as well
as Reznikov's bound (\ref{eqrez}), it is sufficient to show
\[
\sum_{4 \le r_j \le \sqrt{z}} \left( e^{-\frac{r_j^2}{2z}}
\cdot \frac{|r_j^2-z|}{z^{5/2}} + \frac{r_j}{z^2} \right) \cdot
\lambda_j^{1/4}
\rightarrow 0~~~~\text{as}~~z\rightarrow \infty.
\]
By our assumptions on $r_j$ and $z$ {here}, 
we have incorporated the error term
$\frac{\sqrt{r}}{\sqrt{z}}e^{-z/r}$ into the $\frac{r}{z\sqrt{z}}$
error term (inside the  $O$ term of Proposition \ref{prop1})
using simple $\log$ estimates.

If we use the trivial bound of 1 for the exponential term, as well as
$|r_j^2-z| < z$, then the uniform asymptotic (\ref{eqrj}) gives the term
above is
\[
O\left(   \int_{\f{65}4}^{z+1/4}~
\frac{\lambda^{1/4}}{z^{3/2}} + \frac{\lambda^{3/4}}{z^2}  
~dN(\lambda) \right).
\] This can be shown to be $O(1/\sqrt[4]{z})$ using Lemma \ref{lemma4}.
Specifically, this
term can be split into two integrals, which both arise by
estimating sums from Lemma \ref{lemma4}; the first with $\alpha = 1/2$,
the  second with $\alpha = 3/2$, and both with $m_1 = 1/2$.
\end{proof}

\begin{prop}
\[
\lim_{z \rightarrow \infty} ~e^{z} 
\sum_{\sqrt{z} \le r_j \le z^{\frac{14}{25}}}  K_{i r_j}(z) \left| \int_C
\phi_j \right|^2   =
\lim_{z \rightarrow \infty} ~~\frac{\sqrt{\pi}}{\sqrt{2z}}
\sum_{\sqrt{z} \le r_j \le z^{\frac{14}{25}}} e^{-\frac{r_j^2}{2z}}  \cdot
\left| \int_C  \phi_j \right| ^2.
\]  \label{prop4}
\end{prop}

\begin{proof}  As in the previous proposition, the assumptions on $r_j$ and $z$
in the sum of this proposition let us absorb the 
$\frac{\sqrt{r}}{\sqrt{z}}e^{-z/r}$ term into the $\frac{r}{z\sqrt{z}}$
error term (again,  inside the $O$ term of Proposition
\ref{prop1}).

From the error terms in Proposition \ref{prop1}  and
Reznikov's bound (\ref{eqrez}), it is sufficient to show
\[
\sum_{\sqrt{z} \le r_j \le z^{\frac{14}{25}}}
 \left( e^{-\frac{r_j^2}{2z}}
\cdot \frac{|r_j^2-z|}{z^{5/2}} + \frac{r_j}{z^2} \right) \cdot
\lambda_j^{1/4}
\rightarrow 0~~~~\text{as}~~z\rightarrow \infty.
\]
In this region, $r_j^2 >z$ and so $|r_j^2-z|\le r_j^2$.  Consequently, we
need to estimate
\[
\sum_{\sqrt{z} \le r_j \le z^{\frac{14}{25}}}
 \left( e^{-\frac{r_j^2}{2z}}
\cdot \frac{r_j^2}{z^{5/2}} + \frac{r_j}{z^2} \right) \cdot
\lambda_j^{1/4} \]
which is (uniformly) asymptotic to
\[ \frac{1}{z^2}\int_{z+1/4}^{z^{28/25}+1/4} 
\lambda^{3/4} ~dN(\lambda) + 
\frac{e^{\frac{1}{8z}}}{z^{5/2}}\int_{z+1/4}^{z^{28/25}+1/4}
e^{\frac{-\lambda}{2z}}  \lambda^{5/4} ~dN(\lambda).
\]
by (\ref{eqrj}).  (Note there are no issues, such as uniform constants, with
asymptotics here, since an asymptotic was not used in the exponential
factor.) The first integral here can be shown to be $O(1/\sqrt[25]{z})$,
using Lemma \ref{lemma4}, with $\alpha = 3/2$, $m_1 = 1/2$ and
$m_2 = 14/25$.  Hence this term dies off.

The second integral is not quite as easy in this region, and we will need
to use the exponential term.  First, since we are sending $z\rightarrow
\infty$, we can ignore the $e^{\frac{1}{8z}}$ term. An integration by
parts of this integral gives
\[    \frac{1}{z^{5/2}}\left[
\left. e^{-\frac{\lambda}{2z}} \lambda^{5/4} N(\lambda) 
\right|_{z+1/4}^{z^{28/25}+1/4} + \int_{z+1/4}^{z^{28/25}+1/4}
e^{-\frac{\lambda}{2z}} \lambda^{1/4} \left[ \frac{\lambda}{2z}- 5/4 
\right] N(\lambda)~d\lambda  \right] 
\]
Using (\ref{eqnox}), the first term here is 
~$z^{-5/2}
O(e^{-z^{3/25}}\cdot z^{(\frac{5}{4}\cdot\frac{28}{25}+\frac{28}{25})}
+ z^{(\frac{5}{4}+1)})$ which clearly dies as $z \rightarrow \infty$.
Using the uniform bound $N(x) \le c_1 x$ from (\ref{eqnox}) the second
term can be seen to be trivially bounded by
\[
 \frac{c_1}{z^{5/2}} \int_{z+1/4}^{z^{28/25}+1/4}
e^{-\frac{\lambda}{2z}} \lambda^{5/4} \left[ \frac{\lambda}{2z}+ 5/4 
\right] ~d\lambda .
\]With the change of variables $x = \f \lambda  z$, this term
becomes
\[
 \frac{c_1}{z^{5/2}} \cdot z^{9/4}
\int_{1 + \frac{1}{4z}}^{z^{3/25}+\frac{1}{4z}}
~~e^{-\frac{x}{2}}  x^{5/4} \left[ \frac{x}{2}+ 5/4 
\right] ~dx .
\]
Now, the same integral taken over $[1,\infty)$ converges, and this yields
the bound of $O(1/\sqrt[4]{z})$.   
\end{proof}

\begin{prop} With notation as above,
\begin{equation}
0 =\lim_{z \rightarrow \infty} ~e^{z} 
\sum_{z^{\frac{14}{25}} \le r_j \le z}  K_{i r_j}(z) \left| \int_C
\phi_j \right|^2   =    \label{eqprop5}
\lim_{z \rightarrow \infty} ~~\frac{\sqrt{\pi}}{\sqrt{2z}}
\sum_{z^{\frac{14}{25}} \le r_j \le z} e^{-\frac{r_j^2}{2z}}  \cdot
\left| \int_C  \phi_j \right| ^2.
\end{equation}  \label{prop5}
\end{prop}

Even though we are dealing with zero here,  we will
{need} the form of the summand of
the right side to apply a Tauberian theorem later.

\begin{proof}
By Proposition \ref{prop2}, for $r_j$ and $z$ as such, we have
$K_{ir}(z) = ~e^{-z} \cdot O \left( \frac{z^{15/2}}{r_j^{16}} 
\right)$.  Using (\ref{eqrez}), the left hand side 
of (\ref{eqprop5}) can be bounded
by (up to an error, or rather an $O$ constant)
\[
\lim_{z\rightarrow \infty}  ~\left(z^{15/2} \cdot
\sum_{z^{\frac{14}{25}} \le r_j \le z}
\frac{\lambda_j^{1/4}}{r_j^{16}} \right) .
\]     
Since $\sqrt{\lambda_j} \asymp r_j$, the sum here is exactly that of 
Lemma \ref{lemma4} with $\alpha = -31/2$, $m_1 = 14/25$, and
$m_2 = 1$.  Using the bounds of Lemma \ref{lemma4} the above can be seen
to be $O(z^{-3/50})$.

Using Reznikov's bound  (\ref{eqrez}), the right hand side 
of (\ref{eqprop5}) can be seen to be bounded by
\[  \sqrt{\pi/2}
\lim_{z \rightarrow \infty} \frac{e^{1/8z}}{\sqrt{z}} 
\int_{z^{28/25}+
1/4}^{z^2+1/4} e^{-\frac{\lambda}{2z}} \lambda^{1/4}~dN(\lambda).
\]This integral is very similar to one occurring in Proposition
\ref{prop4}, and we evaluate it the same way with the same substitution.
We see the first term dies under integration by
parts.  Using $N(x) \le c_1 x$, we are left with two terms inside the
integral.  Estimating {trivially} (again with the substitution $x = \f \lambda  z$) we need to bound
\[
2c_1 z^{3/4} \int_{z^{3/25}+1/4z}^{z+1/4z}
e^{-x/2} x^2 ~dx .
\]
(Here we have used $2x^2 >x^{5/4}+x^{1/4}$ to simplify the integrand
after an integration by parts; recall we are sending $z \rightarrow
\infty$, so this estimate is trivial in the above region.)
Extending the integration domain to $[z^{3/25},\infty)$, two integrations by parts give us 
$O( c_1 z^{\frac{3}{4}+\frac{6}{25}}e^{-(z^{\frac{3}{25}}/2)})$.  We
clearly win due to the exponential term.  
\end{proof}

This leaves us with the simplest case:

\begin{prop}  With notation as above,
\begin{equation}
0 = \lim_{z \rightarrow \infty} ~e^{z} 
\sum_{r_j > z}  K_{i r_j}(z) \left| \int_C
\phi_j \right|^2   =     \label{eqprop6}
\lim_{z \rightarrow \infty} ~~\frac{\sqrt{\pi}}{\sqrt{2z}}
\sum_{r_j > z} e^{-\frac{r_j^2}{2z}}  \cdot
\left| \int_C  \phi_j \right| ^2
\end{equation} \label{prop6}
\end{prop}

\begin{proof}  We could indeed take care of this case, almost identically to
Proposition \ref{prop5} (with a similar comment as to the start of that
proof),  except for a rather technical piece that 
Lemma \ref{lemma4} only deals with finite sums.  A simple modification can
be made to handle the left side of (\ref{eqprop6}).  
However, in the region $r > z$ it is \textit{much} easier to see
$K_{ir}(z)$ decays rapidly, than using integration by parts 15 times.  If
one integrates by parts \textit{once} the integral representation for
$K_{ir}(z)$ from 8.432, \#5 of \cite{GR} on p.~917 with $x=1$, one
sees that $K_{ir}(z) = O(re^{-\frac{\pi}{2}r})$, for $r>z$.  This is
clear using known $\Gamma$ function estimates.  The rate of decay
$e^{-\frac{\pi}{2}r}$ is clearly better than $ e^{-z}$,
for $r>z$. Further, we
are summing over $r_j$; but this estimate is easy.

The right side of (\ref{eqprop6}) can be handled almost identically to
the right side of (\ref{eqprop5}), as in the  proof of
Proposition \ref{prop5}.  We leave the details to the reader.  
\end{proof}

These estimates prove the theorem.
\end{proof}

We now come to the main result of this section.
\begin{thm}
\[
\sum_{\lambda_j \le x}  |P(\phi_n)|^2  ~\sim~
\frac{\len(C)}{\pi} \sqrt{x}
~~~~\text{as}~ x \rightarrow  \infty.   \label{thm2}
\]
\end{thm}

\begin{proof} This follows easily from a powerful Tauberian theorem.  


Let us construct a measure $H$ on $[0, \infty)$ by putting a point mass
of 
$\left| \int_C  \phi_j \right| ^2 $ at the point  $\lambda_j$.  
(If there is more than one eigenfunction $\phi_j$ corresponding to
$\lambda_j$ we must sum over all, for our point mass.)  Then we can define
the Laplace transform with respect to this measure as (notation as in 
Feller, vol.~II \cite{Feller}):
\[
\omega (y) = \int_0^\infty  e^{-y x} ~H\{ dx \}
=\sum_{j=0}^\infty  e^{-y \lambda_j}  \cdot
\left| \int_C  \phi_j \right| ^2.
\]
By (\ref{thm1laplace}), we see this is finite for each $y>0$.

Applying (\ref{thm1laplace}) with $y = 1/2x$, we see that
\[
\omega (y) ~\sim ~\frac{e^{-y/4}}{\sqrt{y}} \cdot
\frac{\text{len}(C)}{2\sqrt{\pi}}    ~~~~\text{as}~ y \rightarrow 0^+ .
 \]
The function $e^{-y/4}$ is essentially harmless, it is well behaved,
and equal to 1 at
$y=0$.  By Theorem 2 of Feller \cite[p.~421]{Feller}, this gives us (with
$\rho = 1/2$)
\[
U(t) = \int_0^x ~H\{ dx \} ~\sim ~
\frac{\text{len}(C)}{2\sqrt{\pi}\Gamma(3/2)} \sqrt{x}
~~~~\text{as}~ x \rightarrow  \infty.
\]
(This is the Tauberian theorem alluded to above.)

By the definition of our measure, this gives us the asymptotic of the theorem.
\end{proof}

\medskip
As mentioned in the introduction, a very weak bound on the error term 
immediately follows from a Tauberian remainder theorem, precisely \cite[Theorem 3.1]{Kor}.

Notice that we can write the above Laplace transform asymptotic as
\[ F(u) = \omega\(\f 1u\) \sim \f{\len(C)}{2\sqrt \pi}e^{-\f 1{4u}} \sqrt u \]
for $u > 0$, where the asymptotic is as $u \to \infty$.

Since trivially $e^{- \f 1{4u}} = 1 + (e^{- \f 1{4u}}-1)$, we see 
$\epsilon(u)=1-e^{- \f 1{4u}}$ using the notation in \cite{Kor}.  Then the theorem cited
above says that there exists $C_1 \geq 0$, $C_2 > 1$ such that the error term is bounded by
\[ \min_{k \geq k_0} \{ C_1 \f{\len(C)}{2k\sqrt{\pi}} + C_2^k\epsilon\(\f uk\) \} \sqrt u \]
for some $k_0 \in \mathbb N$.

By taking $u$ sufficiently large and using $k=\f 12 \log_{C_2}(u)$, we see the error
term is bounded by
\[ \f{C_1 \len(C)}{\log_{C_2}(u)}\sqrt u = O\(\f {\sqrt u}{\log u}\). \]
With $x$ sufficiently large and $x=u$, we see that our error term is at most this much.

\section{Twisted periods}\label{sec:twistper}

Twisted periods are important in the study of $L$-values as well as in various 
equidistribution problems.  In this section, we show how to include twisting in the
relative trace formula.

First observe that for $n \in \mathbb Z$,
\[ \chi(x) = x^{\f{\pi i n}{\log m}} \]
for $x \in \R$
defines a {\em character} on $C = \Gamma_0 \bs i\R^+$, by which we mean 
\[ \xi \bmx a& \\ & a\inv \emx = \chi(a^2). \]
is a character on the diagonal subgroup of $\PSL_2(\R)$ invariant under $\Gamma_0$.
Now we consider two characters on $C$ given by
\[ \chi_1(x) =  x^{\f{\pi i j}{\log m}}, \, \chi_2(x) =  x^{\f{\pi i k}{\log m}} \]
for some $j, k \in \mathbb Z$.

Now consider the relative
trace formula obtained by integrating
\[ \int_C \int_C K(x,y)\chi_1(x)\chi_2\inv(y)dxdy. \]
From the spectral expansion of the kernel, this clearly equals
\[ \sum h(r_n)P_{\chi_1}(\phi_n)\bar{P_{\chi_2}(\phi_n)} \]
where
\[ P_{\chi_i}(\phi_n) = \int_C \phi_n(t)\chi_i(t)d t. \]
On the other hand, the geometric expansion of the kernel gives as before
\[ \sum_{\Gamma_0 \bs \Gamma / \Gamma_0} I_\gamma(\Phi) \]
where
\[ I_\Id(\Phi) = \int_0^\infty \int_1^{m^2} \Phi(u(ix,iy))\chi_1(x)\chi_2\inv(y)d^\times x d^\times y \]
and
\[ I_\gamma(\Phi) = \int_0^\infty \int_0^\infty \Phi(u(\gamma \cdot ix, iy)) 
\chi_1(x)\chi_2\inv(y) d^\times x d^\times y\]
for $\gamma$ regular.  
One could compute the geometric terms similar to before for an arbitrary $\Phi$, 
but we will content ourselves to the case $\Phi(x) = e^{-tx}$, $t > 0$.

First let us compute the main term.  Put $\mu = \f{\pi i j}{\log m}$ and $\lambda =
\f{\pi i k}{\log m}$.  Then the substitution $u=\f xy$ and $v= xy$ yields
\[ I_\Id(\Phi) = \int_0^\infty \int_1^{m^2} 
\Phi(u + u\inv -2) u^{\f{\mu+\lambda}2}v^{\f{\mu-\lambda}2} d^\times u d^\times v. \]
Note that the integral over $1 \leq v \leq m^2$ is $2 \log m$ if $\mu = \lambda$, i.e.
$j=k$;
otherwise it is
\[ \f 2{\mu - \lambda}\(m^{\mu-\lambda}-1\) \twocase{=}{\f{2\,\len(C)}{\pi i(k-j)}}
{if $k-j$ is odd}{0}{if $k-j\neq 0$ is even} \]
We finish the main term computation similar to the beginning of Section \ref{secfour} to get
\[ I_\Id(\Phi) = 
\begin{cases}
  2\, \len(C)  e^{2t} K_{\f{\mu+\lambda}2}(2t) & \text{ if $j=k$} \\
  \f{4\len(C)}{\pi i(k-j)} e^{2t}K_{\f{\mu+\lambda}2}(2t) & \text{ if $k-j$ is odd} \\
  0 & \text{ if $k-j \neq 0$ is even.}
\end{cases} \]
We will say $\chi_1$ and $\chi_2$ have the {\em same parity} if $k-j$ is even,
and have {\em different parity} if $k-j$ is odd.

Note that since $\chi_1$ and $\chi_2$ are unitary, we may bound the regular geometric
terms by
\[ \left|\sum_{\gamma \neq \Id} I_\gamma(\Phi) \right| \leq 
 \(\sum_{\delta < 2} \f 1{\sqrt{4-\delta^2}}\) \f \pi{t}
+ O\(\f 1{t\sqrt t}\) \]
as in Proposition \ref{firstasymp}.  These estimates immediately give several results
on twisted periods.

\subsection{A single twist}

Our first result on twisted periods, is in the case of a
single twist, i.e., suppose $\chi_1 = \chi_2 = \chi$.
Since $K_\nu(2t)$ and $K_0(2t)$ have the same asymptotics as $t \to \infty$
independent of $\nu$, the asymptotics for $I_\Id(\Phi)$ are the same as in
Section \ref{secfour}.  It follows as in
just as in Proposition \ref{firstasymp} that
\begin{equation}
  \lim_{t \to \infty} e^{t} \sum_{n=1}^\infty K_{ir_n}(t)|P_\chi(\phi_n)|^2 = \f{\len(C)}2. 
\end{equation}
Then we observe that all of the estimates in Section \ref{marksec} apply as before to yield

\begin{thm} For any character $\chi$ of $C$,
\[
\sum_{\lambda_j \le x}  |P_\chi(\phi_n)|^2  ~\sim~
\frac{\len(C)}{\pi} \sqrt{x}
~~~~\text{as}~ x \rightarrow  \infty.   \label{thm3}
\]
\end{thm}

\subsection{Different parity}

Now suppose $\chi_1 \neq \chi_2$ such that $\chi_1$ and $\chi_2$ have different parity.
Then the main geometric term is on the asymptotic order of $t^{-1/2}$, where as any
individual spectral term is on the order of $t\inv$.  This implies we must have
an infinitude of spectral terms.  More precisely we have the following result.
\begin{prop}
  If $\chi_1$ and $\chi_2$ have different parity, then for infinitely many
  $\phi_n$ the periods $P_{\chi_1}(\phi_n)$ and $P_{\chi_2}(\phi_n)$ are
  simultaneously nonvanishing.
\end{prop}

We remark that one cannot use the estimates in Section \ref{marksec} to estimate
the growth of the product of these periods because the spectral terms are 
in general complex.

\subsection{Same parity}

Finally suppose that $\chi_1$ and $\chi_2$ differ but have the same parity.
Then the geometric side is dominated by the exceptional terms.  However, in the
case that there are no exceptional terms, i.e.~$C$ is simple, 
the geometric side is $O(t^{-3/2})$.
Again, we can contrast this with the order of an individual spectral term
to show that infinitely many spectral terms
are nonvanishing.

\begin{prop}
  If $C$ is simple and $\chi_1\neq \chi_2$ have the same parity, then for infinitely many
  $\phi_n$ the periods $P_{\chi_1}(\phi_n)$ and $P_{\chi_2}(\phi_n)$ are
  simultaneously nonvanishing.
\end{prop}

\section{Pairs of geodesics} \label{pairssec}

Let $C_1$ and $C_2$ be distinct closed geodesics on $X$, which we take to be primitive for simplicity.  
We may assume $C_1 = \Gamma_1 \bs i\R^+$ and $C_2 =
\Gamma_2 \bs \tau \cdot i\R^+$ for some $\tau \in \PSL_2(\R)$, where, if $A$ denotes the diagonal subgroup of $\PSL_2(\R)$, $\Gamma_1 = \Gamma_0 = A \cap \Gamma$ and $\Gamma_2 = \tau\inv A \tau \cap \Gamma$.  With the kernel function as before, we now proceed to consider the relative trace formula arising from the
integration
\[ \int_{C_1} \int_{C_2} K(x,y)dxdy. \]
In this case, for {\em any} $\gamma \in \Gamma$, 
the map onto the double coset
$\Gamma_2 \times \Gamma_1 \to \Gamma_2 \gamma \Gamma_1$ given by $(\gamma_2,\gamma_1) 
\mapsto \gamma_2 \gamma \gamma_1$ is injective, i.e., all double cosets are regular.  Hence the geometric side
of the trace formula is
\[ \sum_{\gamma \in \Gamma} \int_{C_1} \int_{C_2} \Phi(u(\gamma \cdot x, y)) dxdy 
= \sum_{\gamma \in \Gamma_2 \bs \Gamma / \Gamma_1} I^\tau_\gamma(\Phi), \]
where
\[ I^\tau_\gamma(\Phi) = \int_0^\infty \int_0^\infty \Phi(u(\gamma \cdot ix, \tau \cdot iy)) d^\times x d^\times y. \]
Note that
\[ I^\tau_\gamma(\Phi) = I_{\tau\inv \gamma}(\Phi). \]
On the other hand, the spectral side is evidently
\[ \sum h(r_n)P_1(\phi_n)\bar{P_2(\phi_n)}, \]
where
\[ P_i(\phi) = \int_{C_i} \phi(x)dx \]
for $i = 1, 2$.  Then, in the same way we obtained Proposition \ref{rtfprop}, one obtains

\begin{prop} (Relative trace formula for pairs of geodesics)   For a function $\Phi$ of sufficiently rapid decay, 
\begin{multline} \label{rtfpairs}
2 \sum_{\delta=\delta(\tau\inv\gamma)<2} \int_{2}^\infty
\f{\Phi(t-2)}{\sqrt{t^2-\delta^2}} \;  
K\( \sqrt{\f{t^2-4}{t^2-\delta^2}} \)dt 
\\
+ \, 
2 \sum_{\delta=\delta(\tau\inv \gamma) > 2} \int_{\delta}^\infty \f{\Phi(t-2)} 
{\sqrt{t^2-4}} \; 
K\(\sqrt{\f {t^2-\delta^2}{t^2-4}}\) dt 
\\
 = \, 
\sum\limits_{n=0}^{\infty} h(r_n)P_1(\phi_n) \bar{P_2(\phi_n)}.
\end{multline}
\end{prop}

It is clear from Section \ref{geometry} that $\delta(\tau\inv \gamma) < 2$ if and only if $C_1$ intersects $C_2$,
in which case this quantity measures the angle of their intersection $\theta$, specifically $\delta=|2\cos\theta|$.  
Otherwise, 
\[ 2 < \delta(\tau\inv\gamma) = 2 \cosh(\dist(\gamma \cdot i\R^+, \tau \cdot i\R^+)). \]
The quantity $\dist(\gamma \cdot i\R^+, \tau \cdot i\R^+))$ may be interpreted on $X$ as the length of a
shortest geodesic segment $\bar \gamma$ starting on $C_1$ and ending on $C_2$ orthogonal to both $C_1$ 
and $C_2$.
Furthermore, all such $\bar \gamma$ come from some $\gamma$ with $\delta(\tau \inv \gamma) > 2$.
Thus the trace formula (\ref{rtfpairs}) expresses the angles of intersection of $C_1$ and $C_2$ and the
lengths of these orthogonal geodesic connectors $\bar \gamma$ in terms of periods of automorphic forms
on $C_1$ and $C_2$.

As in Section \ref{secfour}, specializing the above trace formula to the case where $\Phi(x) = e^{-tx}$ yields
the following.

\begin{prop} (Relative trace formula for pairs of geodesics --- exponential kernel)
\begin{equation}
  \sum_{\gamma \in \Gamma_2 \bs \Gamma / \Gamma_1} e^{2t} K_0\(\f{\delta+2}2t\)
  K_0\(\f{\delta - 2}2t\) 
=  2e^{2t}\sqrt{\f \pi t}\sum_{n=0}^\infty K_{ir_n}(2t)P_1(\phi_n)\bar{P_2(\phi_n)},
 \end{equation}
 where $\delta=\delta(\tau\inv \gamma)$.
 \end{prop}

If we want an analogue of Proposition \ref{firstasymp}, we should estimate the number of double coset representatives
\[ \pi'_\delta(x) = \{ \gamma \in \Gamma_2 \bs \Gamma / \Gamma_1 : \delta(\tau\inv \gamma) < x \}. \]  
We are free to multiply 
$\gamma$ on the left by an element in $\Gamma_2 = \tau A \tau\inv \cap \Gamma$ and on the right by an element
of $\Gamma_1= \Gamma_0$ in choose a set of representatives.  Equivalently, we are allowed multiply 
$\tau\inv \gamma$ on the left and right by elements of $\Gamma_0$.  Hence our counting argument to estimate
$\pi_\delta(x)$ also applies to $\pi'_\delta(x)$ and the rest of the proof of Proposition \ref{firstasymp} goes through
to give

\begin{prop} \label{pairsasymp} As $t \to \infty$, we have the asymptotic
\begin{equation}
2e^{2t}\sqrt{\f \pi t}\sum_{n=0}^\infty K_{ir_n}(2t)P_1(\phi_n)\bar{P_2(\phi_n)} = 
 \(\sum_{\delta < 2} \f 1{\sqrt{4-\delta^2}}\) \f \pi{t}
+ O\(\f 1{t\sqrt t}\).
\end{equation}
\end{prop}
Note that any single term on the left hand side is on the order of $\f 1t$.  Hence
\begin{cor} Suppose $C_1 \cap C_2 = \emptyset$.  Then $P_1(\phi_n)P_2(\phi_n) \neq 0$ for infinitely many $n$.
\end{cor}

\section{Ortholengths} \label{orthosec}

\subsection{Coarse bounds}

In this section we will consider taking $t\to 0^+$ in (\ref{rtfexp}).
We note that the right hand side of (\ref{rtfexp}) contains
$K_{1/2}(2t)$ in the expansion.  This term corresponds to the eigenvalue
zero, where the associated eigenfunction is a constant.  
Our surface may also have exceptional eigenvalues $0 < \lambda < \f 14$.  If this is so, for
such an eigenvalue $\lambda = \frac{1}{4}-\epsilon^2$ the right hand side
above also contains the term $K_{\epsilon}(2t)$.  Furthermore 
$\lambda = \frac{1}{4}$ may also be an eigenvalue, in which case
$K_0(2t)$ also appears on the right.
There is a big difference in asymptotics (for $t$ small) in the
$K$-Bessel functions on the right depending on whether the argument is real
or purely imaginary.  (The case $\lambda = \frac{1}{4}$ corresponding to $K_0$
is in a grey area and has a logarithmic singularity.)

\begin{prop}
\[
2e^{2t}\sqrt{\frac{\pi}{t}}\sum_{n=0}^\infty K_{ir_n}(2t)
|P(\phi_n)|^2 \sim 
\f{\len(C)^2}{\mathrm{vol}(X)} \f{\pi\sqrt{2}}t 
\text{ as }t\rightarrow 0^+.    
\label{propright}
\]
\end{prop}

\begin{proof}
We can take care of $\sum_{r_n > 0}  K_{ir_n}(2t)
|P(\phi_n)|^2$ as one piece.  Let us thus consider
$K_{ir}(z)$ with $r  >0$ and $z>0$ and small.
After integrating 
$2K_{ir}(z) = \int_\mathbb{R} e^{-z\cosh(t)}e^{itr}~dt$ 
by parts
once we arrive at
\begin{equation}
\frac{z}{ir}\int_\mathbb{R} e^{-z\cosh(t)}\sinh(t) 
e^{itr}~dt. \label{eqn2}
\end{equation}
Let $t_0(z)$ be the solution in $t$ of
$\cosh(t) = z\sinh^2(t).$
Since $z$ is very small, $1/z$ is very large, and 
$1/z = \frac{\sinh^2(t_0(z))}{\cosh(t_0(z))}\sim \sinh(t_0(z))$ since
$t_0(z)$ must be relatively large.
Now, the function $e^{-z\cosh(t)}\sinh(t)$ is strictly increasing from
$0$ to $t_0(z)$ and strictly decreasing from $t_0(z)$ to $\infty$, with 
odd symmetry about the origin.  This means we can kill a lot of area
to the right of $0$ by the oscillating factor $e^{itr}$ in (\ref{eqn2});
further the area to the left of $0$ dies also, as long as
\begin{equation}
\frac{\pi}{r}<t_0(z).    \label{eqn3}
\end{equation}
The area that does not cancel is around $\pm t_0(z)$.  If we use
$z\sinh(t) \sim 1$ for $t$ near $t_0(z)$, we see (\ref{eqn2}) is bounded by
$O(\frac{1}{r^2})$.  This analysis depends on whether
(\ref{eqn3}) is true.

For $z$ sufficiently small, $t_0(z) \sim \log(2/z)  $, and so (\ref{eqn3})
becomes approximately
$r >\frac{\pi}{\log(2/z)}.$
We can take $r< r_i$ where $1/4 + r_i^2$ is the first eigenvalue larger
than $1/4$.  Then the part of the sum of (\ref{rtfexp}) over such $r$ is
bounded by  (changing $z$ back to $t$)
\[
O(t^{-1/2})\cdot \sum_{r_n > 0} r_n^{-2} |P(\phi_n)|^2.
\]
We need to show convergence of the series, so we integrate by parts again,
\[2K_{ir}(z)=
\frac{z^2}{r^2}\int_{\mathbb{R}}e^{-z\cosh(t)}\sinh^2(t)e^{itr}~dt -
\frac{z}{r^2}\int_{\mathbb{R}}e^{-z\cosh(t)}\cosh(t)e^{itr}~dt.
\]
The same type of oscillatory analysis applied to both terms here
(separately)
yields
\[
O(t^{-1/2})\cdot 
\sum_{r_n > 0} r_n^{-3} |P(\phi_n)|^2,
\]
which converges by (\ref{eqrez}).

This leaves  only $\lambda = 0$, and possibly exceptional
$\lambda = 1/4 - \epsilon^2$, as well as possibly $\lambda = 1/4$.
 For $t>0$ very small,  $K_0(t)\sim  -\log(t/2)-\gamma_0$, while
for $\epsilon >0$ we have $K_\epsilon(t)\sim \frac{\Gamma(\epsilon)}{2}
(\frac{2}{t})^\epsilon$ (cf.~\cite{AW}; here $\gamma_0$ is Euler's constant.)
 Consequently whatever eigenvalues we have left,
 the $\lambda = 0$ (i.e., the $K_{1/2}(2t)$) term dominates the asymptotics of 
of the spectral side as $t\rightarrow 0$.  
\end{proof}

\begin{prop}
\label{propleft1}
$\pi_{\delta}\(xe^{-\sqrt{\log x }}\) =  O\(\frac{x}{\log x}\)
~\text{as} ~x\rightarrow \infty.$   In particular $\pi_\delta(x) = O(x^{1+\epsilon})$ for
any $\epsilon>0$.
\end{prop}

\begin{proof}
It suffices to consider $\delta > 2$.
We underestimate the regular geometric terms using 
\begin{equation}
  \left[K_0\(\f{\delta+2}2t\)\right]^2 \le K_0\(\f{\delta+2}2 t\)K_0\(\f{\delta-2}2t\).  
  \label{eqn4}
\end{equation}
Since we are taking $t\rightarrow 0$ the $e^{2t}$ factor does not
contribute.  As above, we have
$K_0\(\f{\delta+2}2t\) \sim - \log\(\frac{\delta+2}{4}t\)-\gamma_0,$
as long as the product $\frac{\delta+2}2t$ is sufficiently small (cf.~\cite{AW}).
If we take $\frac{\delta+2}2t$ sufficiently small, by the size of
$\gamma_0$ compared to $\log 2 $ we see we can take
\begin{equation}
K_0\(\f{\delta+2}2t\) \ge -\log \left( \f{\delta+2}2t \right)    .
\label{eqn5}
\end{equation}

Let us now sum over \emph{only} those $\gamma$ for which
\begin{equation}
\delta < \frac{2}{t}e^{-\sqrt{\log(1/t)}} -2   \label{eqn6}
\end{equation}
in (\ref{rtfexp}).  
For such $\delta$, we have
$\log \f 1t \le  \left[K_0\(\f{\delta+2}2t\)\right]^2.$
Thus we see
\[
\sum_{\f{\delta+2}2 < \frac{1}{t} e^{-\sqrt{\log(1/t)}}} 
\log \f 1t  = O\(\f 1t\) \text{ for }t\to 0^+,
\]
where the left hand side comes from an underestimate of the geometric side of
(\ref{rtfexp}) and the right hand side comes from Proposition \ref{propright}.

Setting $x = 1/t$, we get the number of $\f{\delta +2}2 < xe^{-\sqrt{\log x}}$ is
$O\(\f x{\log x}\)$, which gives the above bound.
\end{proof}

\begin{prop}
\label{propleft2}
For $\epsilon>0$, 
$\pi_{\delta}(x) \gg_\epsilon x^{1-\epsilon}
~\text{as}~x\rightarrow \infty.$
\end{prop}

\begin{proof}
We throw away the finite number of $\delta$ in (\ref{rtfexp}) for which
$K_0^2\(\f{\delta-2}2t\) < K_0\(\f{\delta+2}2 t\)K_0\(\f{\delta-2}2t\),$
 as well as the single $K_0$
term on the far left of (\ref{rtfexp}).  By positivity, this gives
\begin{equation}
  \sum_{\delta} K_0^2\(\f{\delta-2}2t\)  \gg \f 1t  ~\text{as}~t 
\rightarrow 0    \label{eqn7}
\end{equation}
by Proposition \ref{propright}, where a finite number of $\delta$ have been
tossed away.

Let $\epsilon, t>0$ and $p>1$  all be small.
Separate the sum of (\ref{eqn7}) into a sum over $\f{\delta-2}2 >1/t^{p}$ and 
a sum over $\f{\delta-2}2\le 1/t^{p}$.
By Proposition \ref{propleft1}, 
Consequently, there is a constant $c_\epsilon>0$ so that
$\f{\delta_n-2}2  \ge c_\epsilon n^{1-\epsilon}.$
By the monotone property of $K_0$, we have
\[
\sum_{\f{\delta-2}2>1/t^p} K_0^2((\delta/2-1)t) \le
\sum_{c_\epsilon n^{1-\epsilon}> 1/t^p} K_0^2(c_\epsilon n^{1-\epsilon}t).
\]
Now $c_\epsilon n^{1-\epsilon}t>1/t^{p-1}$, so for sufficiently small $t$,
$1/t^{p-1}$ is large and $K_0^2(c_\epsilon n^{1-\epsilon}t)$
will be small due to the exponential decay of $K_0(t)$.
From the asymptotic (\ref{K0asymp}), one gets
\[
\sum_{c_\epsilon n^{1-\epsilon}> 1/t^p} K_0^2(c_\epsilon n^{1-\epsilon}t)
< O \left( \sum_{n^{1-\epsilon}>\frac{1}{c_\epsilon t^p}}
\frac{1}{c_\epsilon n^{1-\epsilon}t} 
e^{-2c_\epsilon n^{1-\epsilon}t}   \right)  .
\]
Now $n^{1-\epsilon}>\frac{1}{c_\epsilon t^p}$ implies
$t> c_\epsilon^{-1/p}n^{-(1-\epsilon)/p}$, and so the right hand side
above is
\[
O \left( \sum_{n^{1-\epsilon}>\frac{1}{c_\epsilon t^p}}
\frac{1}{c_\epsilon n^{1-\epsilon}t} 
e^{-2c_\epsilon^{1-1/p} n^{(1-\epsilon)(1-1/p)}}   \right) \le
O \left( \sum_{n^{1-\epsilon}>\frac{1}{c_\epsilon t^p}}
\frac{1}{t^{p-1}} 
e^{-2c_\epsilon^{1-1/p} n^{(1-\epsilon)(1-1/p)}}   \right).
\]
The last term here uses $\frac{1}{t^{p-1}}\le 
c_\epsilon n^{1-\epsilon}t$.  The right side above is clearly seen to be
$O_{\epsilon,p}(1/t^{p-1})$, as the sum converges but depends on
$\epsilon$ and $p$.

We can estimate the remaining terms by
\[
\sum_{\f{\delta-2}2 \le 1/t^p} K_0^2( t) \sim
~(\log(t/2) - \gamma_0)^2  \pi_{\f{\delta-2}2}(1/t^p)
\]
using again the asymptotics of $K_0(t)$.  Here $\pi_{\f{\delta-2}2}(x)$
means the number of $\delta$ such that $\f{\delta-2}2 < x$.

Piecing everything together, we have
\[
(\log(t/2) - \gamma_0)^2  \pi_{\f{\delta-2}2}(1/t^p)
+O_{\epsilon,p}(1/t^{p-1}) \gg \f 1t.
\]
Recall, we have $p$ is barely larger than 1, but the constant depending on
$\epsilon$ and $p$ could be large.  We settle this by writing
\[
\pi_{\f{\delta-2}2}\(1/t^p\) \gg_{\epsilon,p} \frac{1}{t\log^2(t)} 
\implies \pi_{\f{\delta-2}2}\(1/t\) \gg_{\epsilon,p} \frac{1}{t^{1/p}\log^2(t)}.
\]
Setting $1/p = 1-\epsilon$ and $x= 1/t$ proves that $\pi_{\f{\delta-2}2}(x)
\gg_\epsilon \frac{x^{1-\epsilon}}{\log^2(x)}$.  It is easy to see that the same asymptotic
lower bound also holds for $\pi_\delta(x)$, and we may absorb the $\log^2(x)$ into 
$\epsilon$.
\end{proof}

\subsection{An asymptotic}\label{Goodsec}

Let $C_1$, $C_2$, $\tau$, $\Gamma_1$ and $\Gamma_2$ be as in Section \ref{pairssec}.
We present a special case of a result of Good, which is also valid for $\Gamma$ discrete,
cofinite.  One reason \cite{Good} is difficult to understand is that, apart from being
dense, the notation is quite cumbersome, which is at least partially due to the fact that 
he is trying to treat many cases in a uniform way.  (Though even restricted to the case of
smooth compact Riemann surfaces, the formulas in \cite{Good} seem more complicated than ours;
a simple example is the distinction between $\gamma$ and $\gamma\inv$ mentioned below.)

We will explain some of the quantities he considers there in our context, 
using a simplified version of his notation. 
For a hyperbolic $\gamma = \bmx a & b \\ c & d \emx$, let 
$\Lambda^\ell(\gamma) = \f 12 \log |\f{ab}{cd}|$ and $\Lambda^r(\gamma) =
\f 12 \log |\f{ac}{bd}|$ (these are denoted by 
${}_\xi \Lambda^{\ell}_\chi$ and ${}_\xi \Lambda^{r}_\chi$ in \cite{Good}).

For $\gamma \in \Gamma_1 \bs \Gamma / \Gamma_2$ which is regular and non-exceptional 
$(\delta(\tau\inv \gamma) > 2)$, let us write 
\[ N = \bmx a'&b'\\c'&d' \emx = 
\bmx e^{-\f 12\Lambda^\ell(\gamma)} & \\ & e^{\f 12\Lambda^\ell(\gamma)}\emx
\gamma 
\bmx e^{-\f 12\Lambda^r(\gamma)} & \\ & e^{\f 12\Lambda^r(\gamma)} \emx.\]
We define $\nu(\gamma) = |b'| + |d'|$ (cf.~\cite[Lemma1]{Good}).  We see that 
$b'=b|c/b|^{1/2}$ and $d'=d|a/d|^{1/2}$ so $\nu(\gamma) = |ad|^{1/2}+|bc|^{1/2}$.

Define a generalized Kloosterman sum (denoted ${}_\xi^\delta S_\chi^{\delta'}(m,n,\nu)$
in \cite{Good}) by
\[ S_\Gamma(m,n,\nu) =  \sum e \( \f m{\len(C_1)} \log \left|\f{ab}{cd}\right|^{1/2} +
\f n{\len(C_2)} \log\left|\f{ac}{bd}\right|^{1/2} \) \]
where $e(x)=e^{2\pi i x}$ and the sum runs over $\gamma \in \Gamma_1 \bs \Gamma / \Gamma_2$ 
such that $\nu(\gamma) = \nu$.  For $\nu > 1$, this is (up to a bounded number of terms)
 $\sum_{\delta, \delta' \in \{ 0, 1 \}} {}^\delta_\xi S_\chi^{\delta'}(m,n,\nu)$ in Good's notation.

\begin{thm} {\rm (\cite[Theorem 4]{Good})} 
  As $x \to \infty$,
\[ \f 1{\len(C_1)\len(C_2)} \sum_{\nu \leq x} S_\Gamma(m,n,\nu)
\sim \f{\delta_{0,m} \delta_{0,n}}{\pi \mathrm{vol}(X)} x^2 + O(x^{s}), \]
where $\delta_{0,m} = 1$ if $m=0$ and $0$ else, and $s$ is the maximum of $\f 43$ and
$1 + 2ir_n$ for $0 < \lambda_n < \f 14$.
\end{thm}

If in fact there are exceptional eigenvalues $(0 < \lambda_n < \f 14)$, they each give rise 
to highest order terms than the $O(x^{4/3})$, and Good determines what these terms are.

We set $\pi_\delta(x) = \#\{\gamma \in \mathcal O(X; C_1,C_2): \delta(\tau\inv \gamma)<x \}$,
i.e., $\pi_\delta(x)$ is the number of curves $\alpha_\gamma$ in the orthogonal spectrum
such that $\delta(\gamma) = 2 \cosh(\len(\alpha_\gamma)) < x$.

\begin{cor} When $C_1=C_2$, we have
  \[ \pi_\delta(x) \sim \f{\len(C)^2}{\pi\mathrm{vol}(X)} x. \]
  If $C_1 \neq C_2$, we have the asymptotic bound
\[ 
\f 1{\delta(\tau)} \f{{\len(C_1)\len(C_2)}}{{\pi}\mathrm{vol}(X)}x \ll
\pi_\delta(x)
\ll \delta(\tau) \f{{\len(C_1)\len(C_2)}}{{\pi}\mathrm{vol}(X)} x. \]
\end{cor}

This corollary is actually contained in Good's Corollary to Theorem 4, where he asserts something stronger,
though he does not interpret his result in terms of ortholengths.

\begin{proof}
  Applying the theorem with $m=n=0$ gives
  \[ \# \{ \gamma \in \Gamma_1 \bs \Gamma / \Gamma_2 : \nu(\gamma) < x\} \sim
  \f{\len(C_1)\len(C_2)}{\pi \text{vol}(X)} x^2 + O(x^s), \] 
  for some $s < 2$.
  We note that $\nu(\gamma) = \f{\sqrt{\delta(\gamma)+2}+\sqrt{\delta(\gamma)-2}}2$, 
  so $\delta(\gamma) \sim \nu(\gamma)^2$.  It easily follows that
  \[ \# \{ \gamma \in \Gamma_1 \bs \Gamma / \Gamma_2 : \delta(\gamma) < x \} 
  \sim \f{{\len(C_1)\len(C_2)}}{{\pi}\text{vol}(X)} x, \]
  which is the first statment.

  We would like to conclude the same result for $\delta(\tau\inv \gamma)$, but we
  only know how to bound $\delta(\tau\inv \gamma)$ in terms of $\delta(\tau)$ and
  $\delta(\gamma)$.  
  Precisely, the triangle inequality
  \[ \dist(\gamma \cdot i\R^+, \tau \cdot i\R^+) \leq \dist(\gamma \cdot i\R^+, i\R^+)
  + \dist(\tau \cdot i\R^+, i\R^+) \]
  implies $e^{\dist(\gamma \cdot i\R^+, \tau \cdot i\R^+)} \leq e^{\dist(\gamma \cdot i\R^+, 
  i\R^+)} e^{\dist(\tau \cdot i\R^+, i\R^+)}$.  Since $e^x \leq 2\cosh(x) \leq e^x - 1$,
  this yields
  \[ \delta(\tau\inv \gamma) \leq \delta(\tau)\delta(\gamma)-1. \]
  Similarly
  \[ \delta(\gamma) = \delta(\tau(\tau\inv \gamma)) \leq \delta(\tau\inv)\delta(\tau\inv
  \gamma) - 1 = \delta(\tau)\delta(\tau\inv \gamma) - 1. \]
  The bounds in the second statement now follow.
\end{proof}

\end{document}